\documentclass[12pt]{amsart}

\usepackage{amssymb,
latexsym,
amscd, 
amsthm, 
epsfig,
amsmath,
amssymb,
comment, 
fullpage}
\usepackage{color}

\usepackage[colorlinks=true, linkcolor=blue, anchorcolor=blue, citecolor=blue, filecolor=blue, menucolor= blue, urlcolor=red]{hyperref}

\definecolor{brown}{rgb}{.8,.4,.1}
\definecolor{orange}{rgb}{1,.7,0}
\definecolor{gray}{rgb}{.8,.8,.8}
\definecolor{pink}{rgb}{1,.4,.7}
\definecolor{DarkGreen}{rgb}{.13,0.55,0.13}

\definecolor{pink}{rgb}{1,.4,.7}

\numberwithin{equation}{section}

\newtheorem{thm}{Theorem}[section]
 \newtheorem{cor}[thm]{Corollary}
 \newtheorem{lem}[thm]{Lemma}
 \newtheorem{prop}[thm]{Proposition}
 \newtheorem{conj}[thm]{Conjecture}
  \newtheorem{ques}[thm]{Question}
 \theoremstyle{definition}
 \newtheorem{defn}[thm]{Definition}
 \newtheorem{rem}[thm]{Remark}
 \theoremstyle{definition}
 \newtheorem{ex}[thm]{Example}
  \newtheorem{nota}[thm]{Notation}

 \newcommand{\X}{\mathbb{X}}

  \newcommand{\cL}{\mathcal{L}}

 \newcommand{\ZZ}{\mathbb{Z}}

 \newcommand{\XX}{\mathbb{X}}
 \newcommand{\PP}{\mathbb{P}}

\def\move-in{\parshape=1.75true in 5true in}

 \def\P{\mathbb P} 
 \def\ds{\displaystyle}

 \def\X{\mathbb X}
 
 \def\k{\Bbbk}
 
\def\N{\mathbb N}

\def\cocoa{{\hbox{\rm C\kern-.13em o\kern-.07em C\kern-.13em o\kern-.15em A}}}

\usepackage{calrsfs}
\usepackage{calligra}
\usepackage{yfonts}
\DeclareMathAlphabet{\pazocal}{OMS}{zplm}{m}{n}
\newcommand{\La}{\mathcal{L}}

\newcommand{\HS}{\mathbf{HS}}

\newcommand{\WLP}{Weak Lefschetz Property}
\newcommand{\SLP}{Strong Lefschetz Property}

\DeclareMathOperator{\Syz}{Syz}
\DeclareMathOperator{\expd}{exp.dim}
\DeclareMathOperator{\codim}{codim}

\begin{document}

\title[Secant Varieties of the Varieties of Reducible Hypersurfaces in $\P^n$]{Secant Varieties of the Varieties of Reducible Hypersurfaces in $\P^n$} 

\author[M.V.\ Catalisano]{M.V.\ Catalisano}  
\address{Dipartimento di Ingegneria Meccanica, Energetica, Gestionale e dei Trasporti, Universit\`{a} di Genova, Genoa, Italy.}
\email{catalisano@dime.unige.it}

\author[A.V.\ Geramita]{A.V.\ Geramita$^*$} 
\address{Department of Mathematics and Statistics, Queen's University, King\-ston, Ontario, Canada and Dipartimento di Matematica, Universit\`{a} di Genova, Genoa, Italy}
\email{Anthony.Geramita@gmail.com \\ geramita@dima.unige.it}

\author[A.\ Gimigliano]{A.\ Gimigliano} 
\address{Dipartimento di Matematica and CIRAM, Universit\`{a} di Bologna, Bologna, Italy}
\email{alessandr.gimigliano@unibo.it }

\author{B.\ Harbourne}
\address{Department of Mathematics\\
University of Nebraska\\
Lincoln, NE 68588-0130 USA}
\email{bharbour@math.unl.edu}

\author{J.\ Migliore} 
\address{Department of Mathematics, 
 University of Notre Dame, 
  Notre Dame,
IN 46556 \\
 USA}
 \email{migliore.1@nd.edu}

\author{U.\ Nagel}
\address{Department of Mathematics\\
University of Kentucky\\
715 Patterson Office Tower\\
Lexington, KY 40506-0027, USA}
\email{uwe.nagel@uky.edu}

\author[Y.S.\ Shin]{Y.S.\ Shin} 
\address{Department of Mathematics, Sungshin Women's University,  Seoul, 136-742, Republic of Korea}
\email{ysshin@sungshin.ac.kr}

\dedicatory{In fond memory of A.V.\ Geramita, 1942--2016}
     

\begin{abstract}
Given  the space $V=\P^{\binom{d+n-1}{n-1}-1}$ of forms of degree $d$ in $n$ variables, and given
an integer $\ell>1$ and a partition $\lambda$ of $d=d_1+\cdots+d_r$, it is in general an open problem to obtain the dimensions 
of the $(\ell - 1)$-secant varieties $\sigma_\ell(\X_{n-1,\lambda})$ for the subvariety $\X_{n-1,\lambda}\subset V$ of
hypersurfaces whose defining forms have a factorization into forms of degrees $d_1,\ldots,d_r$. 
Modifying a method from intersection theory, we relate this problem to the study of the \WLP \ for a class of graded algebras, based on 
which we give a conjectural formula for the dimension of $\sigma_\ell(\X_{n-1,\lambda})$ for any choice of
parameters $n,\ell$ and $\lambda$. This conjecture gives a unifying framework subsuming all known results.
Moreover, we unconditionally prove the formula in many cases, considerably extending previous results,
as a consequence of which we verify many special cases of previously posed conjectures for dimensions of secant varieties of 
 Segre varieties. 
In the special case of a partition with two parts (i.e., $r=2$), we also relate this problem to a conjecture by Fr\"oberg on 
the Hilbert function of an ideal generated by general forms. 
\end{abstract} 


\date{December 31, 2020}

\thanks{\it *While this paper was being refereed, Tony Geramita passed away. On behalf of Tony's many friends and colleagues 
from all walks of life, the six remaining authors dedicate this paper to him.}

\maketitle

\tableofcontents


\section{Introduction}

Let $S = \k[x_1,\dots,x_n] = \bigoplus_{i \geq 0} [S]_i$, where $\k$ is an algebraically closed field of characteristic zero.  In 1954 Mammana \cite{mammana} introduced the variety of reducible plane curves.  He was seeking to generalize work of C. Segre \cite{segre} (for conics), N. Spampinato \cite{spampianto} (for plane cubics) and G. Bordiga \cite{bordiga} (for plane quartics) as well as other works mentioned in his ample bibliography.

Here we generalize the idea further.  Let $\lambda = [d_1,\dots,d_r]$ be a partition of $d = \sum_{i=1}^r d_i$, which we will write as $\lambda \vdash d$,  where $d_1 \ge d_2 \ge \cdots \ge d_r \ge 1$ and $r \ge 2$.

Consider the variety $\X_{n-1,\lambda} \subset \mathbb P([S]_d) = \mathbb P^{N-1}$ of $\lambda$-reducible forms, where $N = \binom{d+n-1}{n-1}$. That is, 
\[
\X_{n-1,\lambda} = \{ [F] \in \mathbb P^{N-1} \ | \ F = F_1 \cdots F_r \hbox{ for some }0\neq F_i \in [S]_{d_i} \}.
\]
The object of this paper is to study the dimension of the $(\ell-1)$-secant variety of $\X_{n-1,\lambda}$, which we denote by $\sigma_\ell (\X_{n-1,\lambda})$.  
So $\sigma_\ell (\X_{n-1,\lambda})$ is the closure of the union of the linear spans of all subsets of  $\ell$ distinct points of $\X_{n-1,\lambda}$.
We will give a new approach to this problem.

We have
\[
\dim \X_{n-1, \lambda} = \sum_{i=1}^r \binom{d_i + n-1}{n-1}  - r .
\] 
Since all forms in two variables are products of linear forms, we always assume $n \ge 3$, $d \ge r \ge  2$, and $\ell \ge 2$. 
We can (and will) view a general point of $\X_{n-1,\lambda}$ as the product of general forms in $S$ of degrees $d_1,\dots,d_r$ respectively.   

Since, as it is easy to see, no hyperplane contains $\X_{n-1,\lambda}$, 
$\ell\leq N$ general points of $\X_{n-1,\lambda}$ span a linear space of dimension $\ell -1$ (i.e., a
secant $(\ell -1)$-plane), so
by a simple parameter count we have $\dim \sigma_\ell (\X_{n-1,\lambda}) \leq \ell \cdot \dim \X_{n-1,\lambda} + \ell-1$. But
it is possible that $\sigma_\ell (\X_{n-1,\lambda})$ fills its ambient space $\mathbb P^{N-1}$;  this 
clearly happens, for instance, if $\ell\geq N$. 
We combine the two possibilities to obtain an upper bound for the actual dimension of 
$\sigma_\ell(\X_{n-1\lambda})$, typically referred to as the \emph{expected dimension}:
\[
\hbox{exp.dim } \sigma_{\ell} (\X_{n-1,\lambda}) = \min \left \{ \binom{d+n-1}{n-1} -1, \; \ell \cdot \dim \X_{n-1,\lambda} + \ell-1 \right \} .
\]
The {\em defect}, $\delta_\ell$, is the expected dimension minus the actual dimension. When this is positive, we say that $\sigma_\ell (\X_{n-1,\lambda})$ is {\em defective}.  An important part of our work will be to identify when $\sigma_\ell (\X_{n-1,\lambda})$ is defective and to compute the defect.

Secant and join varieties of the Veronese, Segre and Grassmann varieties have been extensively studied.  The recent intense activity in studying these varieties has certainly benefited from the numerous fascinating applications in Communication Theory, Complexity Theory and Algebraic Statistics as well as from the connections to classical problems in Projective Geometry and Commutative Algebra.  (For a partial view of these applications consider the following references and their bibliographies: \cite{AOP}, \cite{AH:1}, \cite{BCS}, \cite{CGG:x}, \cite{CGG:6}, \cite{CS}, \cite{landsberg}, \cite{ravi}, \cite{SS}, \cite{vermeire}.)

However, very little is known about the secant varieties of the varieties of reducible hypersurfaces.  Here also there are useful applications in the study of vector bundles on surfaces and connections to the classical Noether-Severi-Lefschetz Theorem for general hypersurfaces in projective space (see \cite{CCG:1}
\cite{CF}, \cite{Patn}).

The first significant results about the secant varieties of $\lambda$-reducible forms were obtained by Arrondo and Bernardi in \cite{AB} for the special partition $\lambda =[1, \dots ,1]$ (they refer to $\X_{n-1,\lambda}$ for this particular $\lambda$ as the variety of {\it split} or 
{\it completely decomposable} 
forms).  They find the dimensions of secant varieties in this case for a very restricted, but infinite, family of examples.  This was followed by work of Shin \cite{S:1}  who found the dimension of the secant line variety to the varieties of split plane curves of every degree.

This latter result was further generalized by Abo \cite{A}, again for split curves, to a determination of the dimensions of all the higher secant varieties.  Abo also dealt with some cases of split surfaces in $\P^3$ and split cubic hypersurfaces in $\P^n$, for any $n$.

In all the cases considered, the secant varieties  have the expected dimension.  Arrondo and Bernardi have speculated that the secant varieties for split hypersurfaces always have the expected dimension.  We verify this 
for $\sigma_{\ell}(\X_{n-1,\lambda}$) ($\lambda = [1,\dots , 1]\vdash d$) as long as $2\ell \leq n$, which strengthens \cite[Proposition 1.8]{AB}. We also note that their speculation is a special case of our Conjecture \ref{conj:defectivity conj}  (a).

The parameters for this work are $n\geq 3$, $\ell\geq2$ and any partition $\lambda = [d_1,\dots,d_r]$ with $r\geq2$ positive parts. All previous results assume  $d_1=1$ (i.e., the split variety case \cite{A, AB, S:1}) or $n=3$ \cite{CGGS}, or $r=2$ \cite{BCGI, CCG:1}. We extend all of this previous work significantly. See, for example, Theorem \ref{thm: main thm - intro}, as an immediate consequence of which we obtain complete answers in many new cases, including each of the following:
for all $n\gg0$ (in fact $n\geq 2\ell$), fixing any $\ell$ and an arbitrary partition $\lambda$; 
for all $\ell\gg0$ (in fact $\ell\geq \binom{s+n-1}{n-1}$), fixing any $n$ and $d_2\ge\cdots\ge d_r \ge 1$, where $s=d_2+\cdots+d_r$; and
for all $d_1\gg0$ (in fact $d_1\geq (s-1) (n-1)$), fixing any $n$ and $d_2\ge\cdots\ge d_r \ge 1$, and any $\ell$ {\em not} in the interval $(\frac{n}{2},n)$.
We also propose a conjecture (see Conjecture \ref{conj:main conj}), which, if true, gives a complete answer in all remaining cases and which has led us to many of our results.

All approaches to finding the dimension of the secant varieties to a given variety $\X \subset \P^{n-1}$ begin with Terracini's Lemma, including ours.
These all require a good understanding of the tangent space to $\X$ at a general point.  Successful applications of Terracini's Lemma begin by identifying this tangent space as a graded piece of some relatively nice ideal (which we will  call the {\it tangent space ideal}).  To apply Terracini's Lemma, one then needs a way to deal with the sum of tangent space ideals at a finite set of general points of $\X$.  For the Veronese, Segre and Grassmann varieties, the quotient by this ideal sum in the appropriate polynomial ring typically is artinian.  The standard method for dealing with the sum of such ideals (which, per se, have no geometric content) is to pass, using Macaulay duality, to a consideration of the intersection of the {\it perps} of the tangent space ideals (see. e.g. the discussion in \cite{Ger}).  In the classical cases considered above, one obtains a union of special 1-dimensional ideals corresponding to zero dimensional projective schemes.  One then uses geometric methods to get information about the union of the schemes defined by the perps of the tangent space ideals.

This clever use of Macaulay duality had its first notable success with the work of Alexander and Hirschowitz, who completed (after almost one hundred years) the solution of Waring's Problem for general forms (see \cite{AH:1}).  Other work in this direction for these classical varieties can be found in \cite{CCG:1}, \cite{CGG:6},  \cite{CGG:x}, 
\cite{AOP}, \cite{LP}.

In the case of the varieties of reducible hypersurfaces, the method described above no longer works.  In this case, the tangent space ideals already define very nice schemes of dimension $\geq 0$, namely arithmetically Cohen-Macaulay codimension 2 subschemes of $\mathbb P^n$, and their Macaulay duals are artinian!  Thus, one is forced to deal with the sum of the tangent space ideals, i.e., with the {\it intersection} of the  codimension 2 schemes defined by the tangent space ideals at general points.

This is the novelty of our approach: to deal with this intersection we use a version of the so-called diagonal trick from intersection theory, and we show how the so-called Lefschetz properties come in to play in order to study improper intersections. 

As we describe in detail, finding this dimension amounts to viewing the intersection of the aforementioned codimension two subschemes in $\mathbb P^n$ as  the result of consecutive hyperplane sections of their join in $\mathbb P^{n\ell -1}$, where the hyperplanes cut out the diagonal.  
The dimension of the secant variety can then be read off from the Hilbert function in degree $d$ of the coordinate ring of the intersection of such schemes, although ``intersection" must be suitably interpreted in the artinian situation. Algebraically, we are interested in the Hilbert function in degree $d$ of $S/(I_{P_1} + \dots + I_{P_\ell})$  (the $I_{P_i}$ being the tangent space ideals at general points; see Proposition \ref{TangIdeal}), but the geometric notions from intersection theory and hyperplane sections guide our approach.

A key to our method is the observation that we can replace the hyperplanes cutting out the diagonal by truly general hyperplanes. This allows us to compute the dimension of the secant variety $\sigma_{\ell} (\X_{n-1,\lambda})$ in the case where the subschemes meet properly, which occurs precisely when $2 \ell \leq n$. 
In the case of an  improper intersection of the tangent spaces, i.e. $2 \ell > n$, we conjecture that the general hyperplanes induce multiplication maps that all have maximal rank. For a single hyperplane such  behavior has been dubbed the Weak Lefschetz Property in \cite{HMNW}. Assuming this conjectured property of the hyperplane sections, we obtain a formula for the dimension of the secant variety, which is  surprisingly uniform.  
This single formula proposes the dimension for {\em any} choice of $n, \ell$ and $\lambda$.  We will establish it in some cases. It is a conjecture in the rest, but we know of no cases of known results with which it does not agree.

To be more precise, for $0 \le j \le d$ and every given $\ell, n$ and partition $\lambda$ of $d$, we define integers $a_j (\ell, n , \lambda)$ by an explicit formula (see Definition \ref{def:integers a-j}). 
Our formula for the dimension of the secant varieties is the following, which we state as a conjecture so that it can be applied for all $n, \ell$ and $\lambda$, 
in addition to the many cases which we prove below.

\begin{conj}
  \label{conj:main conj} 
Let   $\lambda = [d_1,\ldots,d_r]$ be a partition of $d$  with $r \ge 2$ parts. Then: 
\begin{itemize}
\item[(a)] The secant variety $\sigma_{\ell} (\X_{n-1,\lambda})$ fills its ambient space if and only if there is some  integer $j$ with $s = d_2 + \cdots + d_r \le  j \le d$ such that $a_j (\ell, n , \lambda) \le  0$.

\item[(b)] If $\sigma_{\ell} (\X_{n-1,\lambda})$ does not fill its ambient space, then it has dimension
\begin{align*}
\dim \sigma_{\ell}  (\X_{n-1, \lambda}) = & \ \ell \cdot \dim \X_{n-1, \lambda} + \ell -1    \\[.5ex]
& \hspace*{0.3cm} {\displaystyle -  \sum_{k=2}^\ell (-1)^k \binom{\ell}{k}  \binom{d_1 - (k-1) (d_2 + \cdots + d_r) + n-1}{n-1} }  \\[.5ex]
& \hspace*{0.3cm} {\displaystyle -  \binom{\ell}{2} \binom{2 d_2 - d + n-1}{n-1} - \ell (\ell -1)  \binom{d_1 + d_2 - d+ n-1}{n-1} }. 
\end{align*} 
\end{itemize}
\end{conj}

\noindent The fact that this conjecture is a consequence of the indicated Lefschetz property is shown in Theorem \ref{thm: gen dim formula}. 
Throughout this paper we use the convention that a binomial coefficient  $\binom{a}{b}$ is zero if $a < 0$. Thus, for example, the last and the penultimate term in the above dimension formula are zero if $r \ge 3$.  (A heuristic approach to the formula in Conjecture \ref{conj:main conj}(b) can be found in Remark \ref{sandro rmk}.)

Although stated differently, previous results imply  that Conjecture \ref{conj:main conj}  is true if $n = 3$ and $\lambda = [1,\ldots,1]$  (see \cite{A}), or if $n=3$, $\lambda$ is arbitrary, and $\ell = 2$ (see \cite{CGGS}).  Here, we prove this conjecture in further cases, most of which are summarized in the following theorem. Note that part (b)(i) of the following result was proved in \cite[Theorem 5.1]{CCG:1} using different language.  

\begin{thm}\label{thm: main thm - intro} 
Let $\lambda=[d_1,\ldots,d_r]$ be a partition of $d=d_1+s$ into $r\geq 2$ parts,
where $s=d_2+\cdots+d_r$.
Then Conjecture \ref{conj:main conj} is true in the following cases: 
\begin{itemize}
\item[(a)] $\ell \le \frac{n}{2}$ or $\ell\geq\binom{s+n-1}{n-1}$;
\item[(b)] $r = 2$ and either
\begin{itemize}
\item[(i)] $\ell \le \frac{n+1}{2}$, or 
\item[(ii)] $\lambda = [d-1, 1]$, or
\item[(iii)] $n=3$; and
\end{itemize}
\item[(c)] 
$r\geq3$ and $n\leq \ell\leq 1+\frac{d_1+n-1}{s}$.
\end{itemize}   
\end{thm}   

We prove Theorem \ref{thm: main thm - intro}(a) in Remark \ref{rem:comp with {CGGS}}(ii) and Proposition \ref{prop:non-defective}(c).  
See Theorem \ref{thm:dim when froeberg} for parts (b)(i, iii), Theorem \ref{thm:case d-1, 1} for part (b)(ii), and
Corollary \ref{Thm1.2c} for part (c).

We also show that if $\ell = \frac{n+1}{2}$ (Proposition \ref{prop:upper bound}) 
or $r = 2$ (Proposition \ref{prop:Froeberg gives main conj}), then the number predicted by Conjecture \ref{conj:main conj}  is at least an upper bound for $\dim \sigma_{\ell}  (\X_{n-1, \lambda})$.

Notice that the dimension formula in Conjecture \ref{conj:main conj} involves a series of comparisons, checking whether $a_j (\ell, n , \lambda) >  0$ for all $j = s, s+1,\ldots,d$. 
Accordingly, it is worthwhile to point out more explicitly some of the consequences it suggests. Again, the following is stated as a conjecture even though in the different settings of the above theorem these results are proven.

\begin{conj}\label{conj:defectivity conj} 
Let   $\lambda = [d_1,\ldots,d_r]$ be a partition of $d$  with $r \ge 2$ parts. Then: 
\begin{itemize}
\item[(a)] If $d_1 < d_2 + \cdots + d_r$ (and thus $r \ge 3$), then $\sigma_{\ell}  (\X_{n-1, \lambda})$ is not defective.   
\item[(b)] If $d_1 \ge d_2 + \cdots + d_r$, then the secant variety $\sigma_{\ell}  (\X_{n-1, \lambda})$ is defective if and only if it does not fill its ambient space. 
\end{itemize}
\end{conj}

This conjecture 
highlights the role of the ``partition dividing hyperplane'' $d_1 = d_2 + \cdots + d_r$ 
in the space of partitions with $r$ positive parts, 
as introduced in \cite{CGGS} and discussed here in Remark \ref{PDH}.
Of course, Conjecture \ref{conj:defectivity conj} might be true even if the more
specific formulation given in Conjecture \ref{conj:main conj} is not.
Moreover, while Conjecture \ref{conj:main conj} implies most of Conjecture \ref{conj:defectivity conj},
it is not yet clear that Conjecture \ref{conj:main conj} implies all of Conjecture \ref{conj:defectivity conj};
see Proposition \ref{prop:non-defective} and Remark \ref{NumericalTestingRemark}.
In particular, notice that Conjecture \ref{conj:defectivity conj}(a) is an immediate consequence of 
Conjecture \ref{conj:main conj}, but
we can show that Conjecture \ref{conj:defectivity conj}(b) follows from 
Conjecture \ref{conj:main conj} only in certain cases (see Proposition \ref{prop:non-defective}). 

Our results on defectiveness show unconditionally: 

\begin{thm}\label{thm:defectivity - intro}
Let   $\lambda = [d_1,\ldots,d_r]$ be a partition of $d$  with $r \ge 2$ parts and let $s=d_2+\cdots+d_r$. Then: 
\begin{itemize}
\item[(a)]  If $d_1 < s$ (and hence $r\geq3$) and $2\ell \le n$, then $\sigma_{\ell}  (\X_{n-1, \lambda})$ is not defective.
\item[(b)]  If  $d_1 \ge s$ and $2\ell \le n$,
then $\sigma_{\ell}  (\X_{n-1, \lambda})$ is defective if and only if it does not fill its ambient space. 
\item[(c)] If $\ell\geq n$ and $d_1\geq(n-1)(s-1)$, then $\sigma_\ell(\X_{n-1,\lambda})$ always fills its ambient space,
while if $2\ell \le n$, then $\sigma_\ell(\X_{n-1,\lambda})$ fills its ambient space if and only if one of the following conditions is satisfied:
\begin{itemize}
\item[(i)] $n = 4$, $\ell = 2$,  and  $\lambda \in \{[1,1], [2,1], [1,1,1]\}$ \quad or 
\item[(ii)] $n = 2 \ell \ge 6$ and   $\lambda = [1,1]$.
\end{itemize}
\end{itemize}
\end{thm} 

\noindent We give the proof near the end of Section \ref{sec:improper inters}.

We use our results in the case $r=2$ to study the {\em variety of reducible forms of degree $d$ in $n$ variables}
\[
\X_{n-1, d} = \bigcup_{k=1}^{\lfloor \frac{d}{2} \rfloor} \X_{n-1, [d-k, k]}.
\]
We show that 
\[
\dim \X_{n-1, d}  = \dim \X_{n-1, [d-1, 1]}
\]
and that all other irreducible components of $\X_{n-1, d}$ have dimension that is smaller than the dimension of   $\X_{n-1, [d-1, 1]}$. Thus, one can hope that $\X_{n-1, [d-1, 1]}$ determines the dimension of the secant variety of $\X_{n-1, d}$. Indeed, we establish: 

\begin{thm}\label{thm:intro red forms}
If $2\ell \le n$, then 
\[
\dim \sigma_{\ell} (\X_{n-1, d}) = \dim \sigma_{\ell} (\X_{n-1, [d-1, 1]}). 
\]
Moreover, $\sigma_{\ell}  (\X_{n-1, d})$ is defective if and only it it does not fill its ambient space.
\end{thm}

\noindent We prove this result in Section \ref{sec:variety red forms} as a consequence of Theorem \ref{thm:red forms, small ell}.

Note that the  dimension of $\sigma_{\ell} (\X_{n-1, [d-1, 1]})$ is known for all $\ell, n$ and $d$ (see 
Theorem \ref{thm:case d-1, 1} or  \cite[Proposition 4.4]{BCGI}). Thus, we know exactly when $\sigma_{\ell}  (\X_{n-1, [d-1,1]})$ is defective  (see Theorem \ref{thm:red forms, small ell}). We suspect that $\dim \sigma_{\ell} (\X_{n-1, d}) = \dim \sigma_{\ell} (\X_{n-1, [d-1, 1]})$ is true for all $n$ and $\ell$.  

As another application we consider Segre varieties. We show that results on secant varieties to $\X_{n-1,\lambda}$ imply non-defectivity of secant varieties to a  Segre variety (see Theorem \ref{AOPholds}). 

In Section \ref{sec:inters, diag} we recall the basic facts about the variety of reducible hypersurfaces and how Terracini's lemma is applied. We consider the coordinate ring of the join, which is arithmetically Cohen-Macaulay of codimension $2 \ell$ in $\mathbb P^{n\ell -1}$, with known minimal free resolution and Hilbert function. We discuss how the algebra that determines the dimension of $\sigma_\ell(\X_{n-1,\lambda})$ is obtained by successive hyperplane sections (i.e. reduction by general linear forms). As long as these linear forms are regular elements, the intersection is proper and in Section \ref{sec:proper inters} we give formulas for the dimension and defect.
(As an aside, we point out that the varieties $\X_{n-1,\lambda}$ are not generally
arithmetically Cohen-Macaulay. For example, by direct computation 
they are not for $n=3$ and $\lambda$ either $[1,1,1]$ or $[1,2]$, but we do not know about their secant varieties.)

In Section \ref{sec:secant lines} we summarize our results in the case $\ell = 2$, i.e. for the secant line variety.  Proper intersection corresponds to $n \geq 4$. For the remaining case, $n=3$, we recall the results of \cite{CGGS}. This gives us a bridge from the proper intersections to the improper intersections and gives the idea of how the Lefschetz property is applied in general.

In Section \ref{sec:improper inters} we work out, in general, the connection between the computation of the dimension for arbitrarily large $\ell$, corresponding to improper intersections, and the study of Lefschetz Properties.
 Indeed, based on experiments, we conjecture that the coordinate ring of a certain join variety has enough Lefschetz elements if $2\ell > n$ (see Conjecture \ref{conj:WLP}). If this conjecture is true, then Conjecture \ref{conj:main conj} follows (see Theorem \ref{thm: gen dim formula}). However, Conjecture \ref{conj:main conj} is weaker than the conjecture on the existence 
 of enough Lefschetz elements. 
 
 In Section \ref{sec:r is 2} we focus on the case $r =2$. We show that Conjecture \ref{conj:main conj}  is a consequence of Fr\"oberg's Conjecture on the Hilbert function of ideals generated by generic forms. 
In Section~\ref{sec:variety red forms} we study the variety of reducible forms.
We conclude in Section \ref{AOPstuff} by showing how our results imply cases of conjectures raised in \cite{AOP} about defectivity of Segre Varieties. 


\section{Intersections and the Dimension of Secant Varieties} 
     \label{sec:inters, diag}

After recalling some background and introducing our notation, we lay out our method for computing the desired dimension of a secant variety. It is inspired by a technique from intersection theory. The method will be applied in later sections, where we treat the case of proper and improper intersections separately and carry out the needed computations.  

\begin{nota}
 Let $S = \k[x_1,\dots,x_n] = \bigoplus_{i \geq 0} [S]_i$ be the standard graded polynomial ring, where $\k$ is an algebraically closed field of characteristic zero.  Let $\lambda = [d_1,d_2,\dots,d_r]$ be a {\it partition of $d$ into $r \ge 2$ parts}, i.e., $\lambda \vdash d$, $d_i \in \N$, $d_1 \geq d_2 \geq \dots \geq d_r >0$ and $\sum_{i=1}^r d_i = d$.    

\medskip\medskip
If we set $N = \binom{d+n-1}{n-1}$ then the {\it variety of reducible forms in $[S]_d$ of type $\lambda$} (or {\it the variety of reducible hypersurfaces in $\P^{n-1}$ of type $\lambda$}) is, as noted above: 
$$ 
\X_{n-1,\lambda}:= \{ [F] \in \P([S]_d) = \P^{N-1} \mid F = F_1 \cdots F_r,\ \ \deg F_i = d_i \} \ .
$$ 

The map $([F_1],\ldots,[F_r])\mapsto [F]=[F_1\cdots F_r]$ induces a finite morphism
\begin{equation}
\P([S]_{d_1}) \times \cdots \times \P([S]_{d_r}) \longrightarrow \X_{n-1,\lambda},
\end{equation}\label{SegreToRedFormsMorphism}
and so we have
\begin{equation}\label{dim X}
\dim \X_{n-1,\lambda} = \Big[ \sum_{i=1}^r \binom{d_i + n-1}{n-1} \Big] - r .
\end{equation}

As discussed above, given a positive integer $\ell\leq N$, the 
variety $\sigma_\ell (\X_{n-1,\lambda})$ is the subvariety of $\P^{N-1}$ consisting of the 
closure of the union of secant $\P^{\ell-1}$'s to $\X_{n-1,\lambda}$; for $\ell\geq N$,
$\sigma_\ell (\X_{n-1,\lambda})$ is simply $\P^{N-1}$.
Following the classical terminology, $\sigma_2(\X_{n-1,\lambda})$ is called the {\it secant line variety} 
of $\X_{n-1,\lambda}$ and $\sigma_3(\X_{n-1,\lambda})$ the {\it secant plane variety} of $\X_{n-1,\lambda}$.
\end{nota}

Our main interest in this paper is the calculation of the dimensions of the varieties $\sigma_{\ell}(\X_{n-1,\lambda})$.  

\begin{rem} Notice that for $n = 2$ the question is a triviality since {\em every} hypersurface of degree $d$ in $\P^1$ is reducible of type $\lambda = [1,1,\dots,1] \vdash d$.  Thus, we assume throughout $n \ge 3$.
\end{rem}

The fundamental tool for the calculation of dimensions of secant varieties is the following celebrated result \cite{T:1}.

\begin{prop}[Terracini's Lemma] \label{Terracini}  Let $P_1, \ldots , P_\ell$ be general points on $\X_{n-1,\lambda}$ and let $T_{P_i}$ be the (projectivized) tangent space to $\X_{n-1,\lambda}$ at the point $P_i$.

The dimension of $\sigma_\ell(\X_{n-1,\lambda})$ is the dimension of the linear span of $\bigcup_{i=1}^\ell T_{P_i}$.
\end{prop}


As mentioned above, we have the following definitions.

\begin{defn} \label{exp dim}
The 
{\it{expected dimension}} of $\sigma_\ell(\X_{n-1,\lambda})$, written $\expd (\sigma_\ell (\X_{n-1,\lambda}))$, is
\[
\min\{ N-1, \ell \cdot \dim\X_{n-1, \ \lambda} + (\ell-1) \} .
\]
The {\em defect} of $\sigma_\ell (\X_{n-1,\lambda})$ is
\[
\delta_\ell = \expd (\sigma_\ell (\X_{n-1,\lambda})) - \dim(\sigma_\ell(\X_{n-1,\lambda})) \geq 0.
\]
We say that $\sigma_\ell(\X_{n-1,\lambda})$ is {\em defective} if $\delta_\ell > 0$.

\end{defn}


\begin{rem}\label{defective}
Thus $\sigma_ \ell (\X_{n-1, \lambda})$ is defective if and only if $\dim \sigma_\ell  (\X_{n-1, \lambda}) < N-1$ and $\dim \sigma_\ell  (\X_{n-1, \lambda}) < \ell (\dim(\X_{n-1, \lambda})) + (\ell - 1)$. We will see that in some cases it will be easier to write an expression for $\delta_\ell$ than it will be to show that it is positive. 
\end{rem}

\medskip
Clearly,  to be able to effectively use Terracini's Lemma it is essential to have a good understanding of the tangent spaces to $\X_{n-1,\lambda}$ at general (hence smooth) points.  We do that now.

Let $P = [F] \in \X_{n-1,\lambda}$ be a general point.  Then $F = F_1\cdots F_r$ where the $F_i$ are irreducible forms of degree $d_i $ in $S$.  Let $G_i = F/F_i$, and so $\deg G_i = d-d_i$.  Consider the ideal $I_P \subset S$, where $I_P = (G_1, \ldots , G_r)$.

\begin{prop}\label{TangIdeal} 
In the notation of the preceding paragraph, we have
$$
T_P = \P\Big( [I_P]_d \Big).
$$
\end{prop}

\begin{proof}This proposition is well known and proofs can be found in several places (see e.g. \cite{CCG:1} Prop. 3.2).
\end{proof} 

We refer to the variety in $\PP^{n-1}$ 
defined by $I_P$ as the {\em variety determining the (general) tangent space to $\X_{n-1, \lambda}$}.

As an immediate corollary of Propositions \ref{Terracini} and \ref{TangIdeal} we have the following:

\begin{cor}\label{dim sigma ell}
Let $P_1, \ldots , P_\ell$ be $\ell$ general points of $\X_{n-1,\lambda}$ and $I=I_{P_1} + \cdots + I_{P_\ell}$.  Then
$$
\dim (\sigma_\ell(\X_{n-1,\lambda} )) = \dim_{\k} \left [ I \right ]_d -1.
$$
\end{cor}


This means that $\dim (\sigma_\ell(\X_{n-1,\lambda} ))$ is determined by the Hilbert function in degree $d$ of the ring $S/(I_{P_1}+ \cdots +I_{P_\ell})$, which, when $\sigma_\ell$ does not fill its ambient space, is the coordinate ring of the intersection of $\ell$ varieties, in $\P^{n-1}$, determining tangent spaces to $\X_{n-1, \lambda}$.

\begin{rem}\label{N-1 less}
Let $P_1, \ldots , P_\ell$ be general points on $\X_{n-1,\lambda}$.  Suppose that 
$$
N-1\leq \ell \dim \X_{n-1,\lambda} + (\ell -1);
$$
 i.e., the expected dimension of $\X_{n-1,\lambda}$ is $N-1$.  Then
$$
\sigma_\ell\big( \X_{n-1,\lambda} \big) \hbox{ is defective } \Leftrightarrow \dim_{\k} [S/(I_{P_1} + \cdots + I_{P_\ell})]_d > 0. 
$$
In this case, $\delta_\ell = \dim_{\k} [S/(I_{P_1} + \cdots + I_{P_\ell})]_d$.  
\end{rem}

\medskip\medskip Now that we have seen the ideal that enters into the use of Terracini's Lemma, it remains to give a nicer description of the ideal $I_P$ determining the tangent space at the point $P$.

\begin{prop}\label{nicer} 
Let $P$ be a general point of $\X_{n-1,\lambda}$, $P = [F_1\cdots F_r]$ where $\deg F_i = d_i$.  Put $F = F_1\cdots F_r$. Then we have
$$
I_P =  (F/F_1, \ldots ,F/F_r) = \bigcap _{1 \leq i < j \leq r} (F_i, F_j) .
$$
\end{prop}

\begin{proof} The first equality is given in Proposition \ref{TangIdeal}.  
 The second equality is well known (see for example \cite{PS:1}, Thm. 2.3).
\end{proof} 

Thus the ideal $I_P$, for a general point $P \in \X_{n-1,\lambda}$, is of codimension 2 in $S$ and is a finite intersection of interrelated codimension 2 complete intersection ideals.  Such ideals (and their  generalization to the situation where the complete intersection ideals have higher codimension) have been studied in several papers for many different reasons (see e.g. \cite{GHM:1}, \cite{GHMN:1}, \cite{GMS:1} and \cite{CGVT:1}). 

We now derive the graded minimal free resolution of the ideal $I_P$. 

\begin{lem}\label{lem3.1}
Let $R=\k[Y_1,\dots,Y_r]$,  $M=Y_1\cdots Y_r$, $M_i=M/Y_i$, where $r \ge 2$.   If $I=(M_1,\dots,M_r)$ then $I ={\ds\bigcap_{1\le i<j\le r}}(Y_i,Y_j)$ and the minimal graded free resolution of $I$ is
$$
0 \to R^{r-1}(-r) \overset{A}{\to} R^r(-(r-1)) \to R \to R/I \to 0.
$$
\end{lem}

\begin{proof} Consider the matrix  $A$, defined by 
$$
A^t=\begin{bmatrix}
Y_1 & -Y_2 & 0 & \cdots & 0 \\
Y_1 & 0 & -Y_3 & \cdots & 0 \\
\vdots &  \vdots &\vdots & \vdots & \vdots \\
Y_1 & 0 & 0 & \cdots & -Y_r
\end{bmatrix}_{(r-1)\times r\ .}  
$$
The ideal generated by the maximal minors of $A$ is $I$. It has codimension 2 in $R$.   Thus, the claim follows from the Hilbert-Burch theorem. 
\end{proof} 


\begin{rem} \label{info about S/J}
Let $R=\k[Y_1,\dots,Y_r]$ be as above and let $S=\k[x_1,\dots,x_n]$. Let $F_1,\dots,F_r$ be general homogeneous polynomials in $S$ of degrees $d_1,\dots,d_r$ respectively. Let $F=\prod^r_{i=1}F_i$, $\deg F=d=\sum^r_{i=1}d_i$ and $G_i=F/F_i$, $\deg G_i = d-d_i$, for $1\le i\le r$. Let
$$
\varphi : R \to S
$$
be defined by $\varphi(Y_i)=F_i$. Then, with $I$ as in Lemma \ref{lem3.1}, 
$\varphi(I)=(G_1,\dots,G_r)=\bigcap_{1\le i<j\le r} (F_i,F_j)=J$.

By the generality in the choice of the $F_i$,  $J$ is Cohen-Macaulay of codimension $2$ (see Proposition \ref{nicer}) and its Hilbert-Burch matrix has transpose
$$
\begin{bmatrix}
F_1 & -F_2 & 0 & \cdots & 0 \\
F_1 & 0 & -F_3 & \cdots & 0 \\
\vdots &  \vdots &\vdots & \vdots & \vdots \\
F_1 & 0 & 0 & \cdots & -F_r
\end{bmatrix} _{(r-1)\times r \ .}
$$
The minimal free graded resolution of $S/J$ is
\begin{equation} 
       \label{MFR}
0 \to S^{r-1}(-d)  {\to} \bigoplus_{i=1}^r S(-(d-d_i)) \to S \to S/J \to 0. 
\end{equation}
It is a simple consequence of this resolution that the artinian reduction of $S/J$ is level with socle degree $d-2$ and Cohen-Macaulay type $(r-1)$.  \qed
\end{rem}

\begin{rem}
Observe that $J$ is the ideal $I_P$ for the point $P = [F]$ on $\X_{n-1, \lambda}$ if we assume $d_1 \ge d_2 \ge \cdots \ge d_r$.   
In this paper the main goal is to consider  sums of such ideals, i.e. ideals of the form $(G_1,\dots,G_r)$ arising from general forms of prescribed degrees as described above, and to compute the dimension of the component in degree $d$. The fact that they correspond to general points on $\X_{n-1,\lambda}$ is not needed for most of our computations. Thus, to emphasize the focus on the ideals rather than the points, we will write $I_{(1)} + \dots + I_{(\ell)}$ in place of $I_{P_1} + \dots + I_{P_\ell}$ when the geometric context is not needed, and retain the latter only when the geometry is important (e.g., Remark~\ref{sandro rmk}).
\end{rem}

\begin{rem}
    \label{HS} 
Let us recall a few results about the Hilbert series of a standard graded ring.
 
Let $A= \oplus_{i=0}^{\infty} [A]_i$.  The {\it Hilbert series of $A$} is the formal power series
$$
\HS(A) = \sum_{i=0}^{\infty} (\dim [A]_i)t^i .
$$

It is a simple matter to show the following two facts, which we will use often in what follows:

$(a)$ If $L$ is a linear non-zerodivisor in $A$ then
$$
\HS(A/LA) = (1-t)\HS(A).
$$

$(b)$  ${\ds \HS(\k[x_1, \ldots ,x_n]) = \frac{1}{(1-t)^n}}$ and  ${\ds \HS(\k[x_1, \ldots ,x_n](-a)) = \frac{t^a}{(1-t)^n}}$. 

\medskip 

Of course $(b)$ is a simple consequence of $(a)$.  

$(c)$  We can apply these observations to the minimal free resolution \eqref{MFR} in Remark \ref{info about S/J} in order to conclude that
\begin{equation}
    \label{eq:hilbSeries tangent}
\HS(S/J) = \frac{1}{(1-t)^n}\Big[ 1 - \sum_{i=1}^r t^{d-d_i} + (r-1)t^d\Big] \ .
\end{equation} 

$(d)$ If $A$ and $B$ are graded $\k$-algebras, then 
\[
\HS (A \otimes_{\k} B) = \HS (A) \cdot \HS (B). 
\] 
\end{rem}
\medskip



Consider a partition $\lambda=[d_1,\dots,d_r]$, $\lambda \vdash d$.   In the polynomial ring $\k[x_1,\dots,x_n]$ choose general homogeneous forms 
$F_1,\dots,F_r$ of degrees $d_1,\dots,d_r$ and, as in Remark \ref{info about S/J}, let $F=\prod^r_{i=1} F_i$, $G_i=F/F_i$ and $I=(G_1,\dots,G_r)=\bigcap_{1\le i<j\le r} (F_i,F_j)$.

\medskip

Inasmuch as we are interested in the secant variety $\sigma_\ell(\X_{n-1,\lambda})$ we form $\ell$ sets of general polynomials as above in $\k[x_1,\dots,x_n]$. Call the elements of the $j$-th set
$$
\{F_{j,1},\dots,F_{j,r}\},
$$
where $\deg F_{j,k}=d_k$.  As in Remark \ref{info about S/J}, for $1\leq j \leq \ell$ form $M_j = \prod^r_{i=1}F_{j,i} $, and  $G_{j,1},\dots,G_{j,r}$ where $G_{j,k}= M_j/F_{j,k}$.

Set
$$
I_{(j)}=(G_{j,1},\dots,G_{j,r})=\bigcap_{1\le i<k\le r} (F_{j,i}.F_{j,k}), \ \ \ \   1\leq j \leq\ell .
$$
Notice that each of the quotients $S/I_{(j)}$ has the same Hilbert function and minimal free resolution as that of $S/J$ given in Remark~\ref{info about S/J}. Furthermore, each ideal $I_{(j)}$ defines a variety determining the tangent space to $\X_{n-1, \lambda}$ at the point 
$P_j = [F_{j1}F_{j2}\cdots F_{jr}]$.  

We can perform the same construction as above, but this time choosing each set of $r$ general polynomials in different polynomial rings, i.e., consider 
$\{F_{j,1},\dots,F_{j,r}\}$ as polynomials in the ring $\k[x_{j,1},\dots,x_{j,n}]$. 
We can form the sum of these ideals (extended) in
$$
T=\k[x_{1,1},\dots,x_{1,n},\dots,x_{\ell,1},\dots,x_{\ell,n}] \ ,
$$
setting $\tilde I=I_{(1)}^{(e)}+\cdots+I_{(\ell)}^{(e)}$ (i.e., the sum of the extended ideals).

\begin{thm}
   \label{thm:join as tensor product}
The ring 
$$
B = T/\tilde I \cong S/I_{(1)}\otimes_{\k} \cdots \otimes_{\k} S/I_{(\ell)}
$$
is Cohen-Macaulay of dimension $\ell(n-2)$. Its  minimal graded
free resolution over $T$ is the tensor product (over $\k$) of the minimal graded free resolutions of $S/I_{(j)}$ over $S$ for $1\le j\le \ell$.
\end{thm}

\begin{proof}
This is a consequence of the K\"unneth formulas. See \cite[Lemma 3.5]{MNP} and its proof.
\end{proof}

Note that $B$ is the coordinate ring of the join of $\ell$ varieties, each of which has codimension 2 in $\P^{n-1}$, so their join is in $\P^{n\ell -1}$. The so-called diagonal trick  gives 
\[
S/(I_{(1)} + \cdots + I_{(\ell)}) \cong B/\Delta B, 
\]
where the diagonal $\Delta$ is generated by the $(\ell - 1) n$ linear forms $x_{1,j} - x_{i, j}$ with $1 < i \le \ell$ and $1 \le j \le n$. Observe that the saturation of $I_{(1)} + \cdots + I_{(\ell)}$ defines the intersection of the indicated varieties in $\P^{n-1}$, provided this intersection is not empty. 

A key to our approach is the fact that replacing the linear forms generating the diagonal by truly general linear forms gives a quotient ring with the same Hilbert function as \linebreak $S/(I_{(1)} + \cdots + I_{(\ell)})$. To illustrate the idea, fix a polynomial ring $R$ in $m$ variables,  and  let $L \in R$ be a general linear form.  Since we have a surjection $R \rightarrow R/( L )$, if $\{ F_1,\dots,F_t \}$ is a set of general forms in $R$ of degrees $d_1,\dots,d_t$, then the restriction, $\{ \bar F_1 \dots, \bar F_t \}$, to $R/( L )$  can be viewed again as a set of general forms of degrees $d_1,\dots,d_t$ in $m-1$ variables.  Furthermore, given a prescribed construction of an ideal in $m$ variables using general forms of prescribed degrees, the restriction to $R/( L )$ of this ideal can be viewed as an application of the same construction to an ideal of general forms of the same degrees but in $m-1$ variables.  

In our setting, if $[I_P]_d = [(G_1,\dots,G_r)]_d$ is the vector space determining the tangent space to $\X_{n-1,\lambda}$ at a general point  $P$ (see Proposition \ref{TangIdeal}), then $[\bar I_P]_d = [(\bar G_1,\dots,\bar G_r)]_d$ is the degree $d$ component of an ideal that determines the tangent space at a general point of the variety $\X_{n-2,\lambda}$. The analogous statement also holds for an ideal of the form $I_{P_1}+\dots + I_{P_\ell}$.

Returning to the above notation, let $\La$ be a set of $(\ell-1)n$ general linear forms in $T$.  Then we have the following useful observation. 

\begin{lem} 
      \label{sop} 
The algebras $S/(I_{(1)} + \cdots + I_{(\ell)})$ and $B/{\La}B \cong T/({\La}, {\tilde I})$ have the same Hilbert series.  I.e.
$$
\HS (S/(I_{(1)} + \cdots + I_{(\ell)})) = \HS (B/{\La}B) \ .
$$
\end{lem}

\begin{proof} Each ideal $I_{(j)} \subset S$ corresponds to a choice of a general point on $\X_{n-1,\lambda}$.  Thus, it is generated by the $r$ products of $r-1$ distinct forms that are created using $r$ general forms of degrees $d_1, d_2, \ldots , d_r$ in variables $x_1, x_2, \ldots , x_n$ (see Proposition \ref{nicer}).  The same is true for the summand $I^{(e)}_{(j)}$ of $\tilde I$, although these forms are in a new set of variables.  Since the linear forms in $\La$ are general, the residue classes of the forms defining $I^{(e)}_{(j)}$ modulo $\La$ are again general forms in $S$.  It follows that the image $\overline{I^{(e)}_{(j)}}$ of $I^{(e)}_{(j)}$ in $T/{\La}T \simeq S$ also corresponds to a general point on $\X_{n-1,\lambda}$.  Thus, the ideals $I_{(1)} + \cdots + I_{(\ell)}$ and $\overline{I_{(1)}^{(e)} } + \cdots + \overline{I_{(\ell)}^{(e)} }$ have the same Hilbert function and hence the same Hilbert series. 
\end{proof}

\begin{rem}
Lemma \ref{sop} is, in a sense, the key to the results in this paper. In combination with Corollary \ref{dim sigma ell} it shows that computing the dimension of $\sigma_\ell  (\X_{n-1,\lambda})$, for {\em arbitrary} $n, \ell$ and $\lambda$, is equivalent to finding the coefficient of $t^d$ in the Hilbert series of $B/\mathcal L B$.  We emphasize here that the only requirement for the linear forms in $\mathcal L$ is that they be general. We do not need them to be regular elements. The next section will handle the case where they are regular elements, and  subsequent sections deal with the case where some of the linear forms are not  regular elements. 
\end{rem}

For our discussion of  $\dim_{\k} [B/{\La}B]_d$ and more generally the Hilbert series of $B/{\La}B$, it is helpful to consider two cases. We refer to them as {\it proper} and {\it improper} intersections.

Consider varieties $V_1,\ldots,V_s \subset \PP^{n-1}$. Then their intersection 
 is defined by the saturation of $I_{V_1} + \cdots + I_{V_s}$ and  satisfies 
\[
\codim (I_{V_1} + \cdots + I_{V_s}) \le \codim I_{V_1} + \cdots + \codim I_{V_s}.  
\]
Abusing notation slightly (in the case where $\codim (I_{V_1} + \cdots + I_{V_s}) = n$, i.e., the intersection is  the empty set), we say 
 that the varieties $V_1,\ldots,V_s \subset \PP^{n-1}$ {\em intersect properly} if 
\[
\codim (I_{V_1} + \cdots + I_{V_s}) = \codim I_{V_1} + \cdots + \codim I_{V_s}.  
\] 
Otherwise, they {\em intersect improperly}.

In particular, this means that, fixing $n$ and the partition $\lambda$,  if the intersection of the  varieties $V(I_{(1)}),\ldots,V(I_{(\ell)})$ is the empty set  for some $\ell = \ell_0$, then these varieties intersect improperly for all $\ell > \ell_0$.
\smallskip

We close this section with a fact we will have opportunities to apply later.
We can partially order partitions of an integer $d>0$ as follows. Given partitions $\lambda_1=[d_1,\ldots,d_p]$ 
and $\lambda_2=[e_1,\dots, e_q]$ of the same integer $d>0$, 
write $\lambda_1\geq\lambda_2$ if for each $i\geq0$ we have $\sum_{j\leq i}d_j\geq \sum_{j\leq i}e_j$
(where we regard $d_j$ and $e_j$ as being 0 if $j$ is out of range).
Write $\lambda_1>\lambda_2$ if $\lambda_1\geq\lambda_2$ but $\sum_{j\leq i}d_j> \sum_{j\leq i}e_j$
for some $i$. So, for example, if $q > p$, then either $\lambda_1$ and $\lambda_2$ are incomparable (as happens with $\lambda_1 = [4,3,1]$ and $\lambda_2 = [5,1,1,1]$) or $\lambda_1 > \lambda_2$ (as happens with $\lambda_1 = [5,2,1]$ and $\lambda_2 = [5,1,1,1]$). 

\begin{lem} \label{lem:compare binomials}
Let $\lambda_1=[d_1,\ldots,d_p]$ and $\lambda_2=[e_1,\dots, e_q]$ be partitions of the same integer $d$
with $\lambda_1>\lambda_2$. If $n\geq 3$, then $\dim \X_{n-1,\lambda_1}>\dim \X_{n-1,\lambda_2}$.
\end{lem}

\begin{proof}
Let $u$ be the least $i$ such that $d_i>e_i$. (There must be such an $i$ since
$d_i\leq e_i$ for all $i$ implies $\sum_{j\leq i}d_j\leq \sum_{j\leq i}e_j$ for all $i$.)
Note that if $u>1$, then 
$e_{u-1}=d_{u-1}\geq d_u>e_u$ 
and $\sum_{j\leq u}d_j> \sum_{j\leq u}e_j$.

Next, let $v$ be the least $i>u$ such that $\sum_{j\leq i}d_j= \sum_{j\leq i}e_j$.
(There must be such an $i$ since both sums eventually are equal to $d$.)
Note that if $v<q$, then $e_v>e_{v+1}$. This is because if $v<q$, then
(by definition of $v$ and the fact that $\sum_{j\leq u}d_j> \sum_{j\leq u}e_j$) we have
$\sum_{j\leq v-1}d_j > \sum_{j\leq v-1}e_j$, but $\sum_{j\leq v}d_j = \sum_{j\leq v}e_j$,
so $e_v>d_v$, and $\sum_{j\leq v+1}d_j \geq \sum_{j\leq v+1}e_j$, 
so $d_v\geq d_{v+1}\geq e_{v+1}$.

Now let $\lambda_3=[f_1,\ldots,f_r]$ where $f_u=e_u+1$, $f_v=e_v-1$, and
otherwise $f_j=e_j$. Then $f_j$ is nondecreasing since $e_j$ is and
either $u=1$ or $f_{u-1}=e_{u-1}\geq e_u+1=f_u$, and 
either $v=q$ or $f_v=e_v-1\geq e_{v+1}=f_{v+1}$. 
Moreover, $\sum_{j\leq i}d_j \geq\sum_{j\leq i} f_j\geq\sum_{j\leq i} e_j$ is true for all $i$.
It holds for $i<u$ since $f_j=e_j=d_j$ in this range.
It holds for $u\leq i<v$ since $\sum_{j\leq i}d_j >\sum_{j\leq i} e_j$
but $1+\sum_{j\leq i}e_j =\sum_{j\leq i} f_j$ in this range.
And it holds for $i\geq v$, since $\sum_{j\leq i}e_j =\sum_{j\leq i} f_j$
in this range.

Thus $\lambda_1\geq\lambda_3>\lambda_2$, so it suffices by induction to show
$\dim \X_{n-1,\lambda_3}>\dim \X_{n-1,\lambda_2}$. Writing each $f_j$ in terms of $e_j$, this is equivalent to showing
$\binom{e_u+1+n-1}{n-1}+\binom{e_v-1+n-1}{n-1}-2>\binom{e_u+n-1}{n-1}+\binom{e_v+n-1}{n-1}-2$.
This in turn is equivalent to
\begin{align*}
\textstyle{\binom{e_u+1+n-2}{n-2}} & = \textstyle{\binom{e_u+n-1}{n-2}} \\
& = \textstyle{\binom{e_u+n}{n-1}-\binom{e_u+n-1}{n-1}>
\binom{e_v+n-1}{n-1}-\binom{e_v-1+n-1}{n-1}=\binom{e_v-1+n-1}{n-2}=\binom{e_v+n-2}{n-2},}
\end{align*}
which is true because $\binom{j + n-2}{n-2}$ is a strictly increasing function of $j \ge 0$ if $n-2 \ge 1$.
\end{proof}

We now have:

\begin{cor}\label{cor:dim comparison}
Let $\lambda=[d_1,\ldots,d_r]$ be any partition of $d$ with $r\geq 2$,
and let $\lambda_2=[d_1,1,\ldots,1]$ also be a partition of $d$. Assume $n \geq 3$.
If $\lambda$ is neither $\lambda_2$ nor $[d-1, 1]$, then
\[
\dim \X_{n-1, \lambda_2} <\dim \X_{n-1, \lambda} < \dim \X_{n-1, [d-1, 1]}.
\]
\end{cor}

\begin{proof}
This follows from Lemma \ref{lem:compare binomials} and the fact that $\lambda_2<\lambda<[d-1, 1]$.
\end{proof}

Notice that a result analogous to Lemma \ref{lem:compare binomials} is not true for the lexicographic order. For example, $\lambda_1 = [5,4,1,1,1,1] > [5,3,3,2] = \lambda_2$ in the lexicographic order, but $\dim \X_{2, \lambda_1} = 42 < 43 = \dim \X_{2, \lambda_2}$.  Observe that $\lambda_1$ and $\lambda_2$ are not comparable in the partial order used in Lemma \ref{lem:compare binomials}.


\section{Proper Intersections} 
     \label{sec:proper inters}
     
In this section we focus on the case where the varieties determining tangent spaces to $\X_{n-1, \lambda}$ at $\ell$ general points meet properly. Our main result is Theorem \ref{thm:dim formula}. The case of improper intersections is the subject of a later section. 

We first show that the  $\ell$ varieties determining tangent spaces intersect properly if $\ell$ is small enough. 

\begin{prop} 
     \label{prop:proper intersection} 
Assume $2\ell \le n$. Then:    
\begin{itemize}
\item[(a)] The $(\ell-1)n$ general linear forms in $\La$ are a $B$-regular sequence. 

\item[(b)] The varieties defined by $I_{(1)},\ldots,I_{(\ell)}$ intersect properly, that is, 
\[
\codim (I_{(1)}+\cdots+I_{(\ell)})= 2 \ell.
\]
\end{itemize}  
\end{prop}

\begin{proof} By Theorem \ref{thm:join as tensor product}, the algebra $B$ is Cohen-Macaulay of dimension $\ell (n-2)$. The assumption on $\ell$ guarantees $(\ell - 1) n \le \ell (n-2)$. Hence $\La$ is a regular sequence and $B/\La B$ has dimension $n - 2 \ell$. Now Lemma \ref{sop} gives $
\codim (I_{(1)}+\cdots+I_{(\ell)} )= 2 \ell$.
\end{proof} 

\begin{rem} If $\ell \le \frac{n}{2}$, then the minimal free graded resolution of $S/(I_{(1)}+\cdots+I_{(\ell)})$
has the same graded Betti numbers as the minimal free graded resolution of $T/\tilde I$ since forming a quotient by factoring with a regular sequence does not change the graded Betti numbers of the resolution modules.  
\end{rem}

\begin{rem} 
       \label{formula}
Using the isomorphism of graded modules (see Theorem \ref{thm:join as tensor product})
\[
T/\tilde I \cong S/I_{(1)} \otimes_{\k} \cdots \otimes_{\k} S/I_{(\ell)},
\]
it follows  that $\HS(T/\tilde I)=(\HS(S/J))^\ell$, where $J$ is as given in Remark \ref{info about S/J}.

Furthermore, if  $2\ell\le n$, Proposition \ref{prop:proper intersection}(a), Remark \ref{HS}(d), and Lemma \ref{sop} give  
\[ 
\begin{array}{lllllllllllllllll}
 \HS(S/(I_{(1)}+\cdots +I_{(\ell)}))
& = &(1-t)^{n(\ell-1)} \cdot  \HS(T/\tilde I) \\
& = & (1-t)^{n(\ell-1)} \cdot [\HS( S/J)]^\ell.   
\end{array}
\]
Putting this together with Equation \eqref{eq:hilbSeries tangent} in Remark \ref{HS} we obtain 
\[
\HS(\k[x_1,\dots,x_n]/(I_{(1)}+\cdots +I_{(\ell)})) = \frac{1}{(1-t)^n} \left [  1-\sum^r_{i=1} t^{d-d_i} +(r-1)t^d       \right] ^\ell .
\]

Rewriting this last expression we have, if $2\ell \leq n$, that
\begin{equation*} \label{key eqn}
\begin{array}{llllllllllllllllllllllll}
\HS (I_{(1)}+\cdots +I_{(\ell)})
=  \ds \frac{1}{(1-t)^n} - \frac{1}{(1-t)^n}\left [  1-\sum^r_{i=1} t^{d-d_i} +(r-1)t^d       \right] ^\ell.
\end{array}
\end{equation*}
\end{rem}

\medskip\medskip

If we now put together Corollary \ref{dim sigma ell} and Remark \ref{formula} we obtain: 

\begin{thm}
      \label{HS and dim sigma ell} 
Let $\lambda\vdash d$, $\lambda = [d_1, \ldots , d_r]$ and suppose that $2\ell \leq n$. Put 
\[
A = \k[x_1,\dots,x_n]/(I_{(1)}+\cdots +I_{(\ell)}). 
\] 
Then 
\[
\HS (A) = \frac{1}{(1-t)^n}\left [  1-\sum^r_{i=1} t^{d-d_i} +(r-1)t^d       \right] ^\ell.
\] 
Moreover, if  $a_d$ denotes the coefficient of $t^d$ in $\HS (A)$, then 
$$
\dim \sigma_\ell (\X_{n-1,\lambda}) = \binom{d+n-1}{n-1} - a_d - 1 .
$$
\end{thm}

We now compute the coefficient $a_d$, which gives the main result of this section. Although it is not obvious that the right-hand side of 
the formula for $\dim \sigma_\ell(\X_{n-1,\lambda})$ given in
Theorem \ref{thm:dim formula} 
is less than or equal to $\binom{d+n-1}{n-1} - 1$,
this follows from the fact that the right-hand side is
$\binom{d+n-1}{n-1} - a_d - 1$, since $a_d = \dim_{\k} [A]_d$.

\begin{thm}\label{thm:dim formula} 
Let $\lambda\vdash d$, $\lambda = [d_1,\ldots,d_r]$   with $r \ge 2$.  If $2\ell \le n$ then:
\begin{eqnarray*}
  \dim \sigma_{\ell}  (\X_{n-1, \lambda})  & = &  \ell \cdot \dim \X_{n-1, \lambda} + \ell -1  \\
&& \ - \sum_{k=2}^\ell (-1)^k \binom{\ell}{k}  \binom{d_1 - (k-1) (d_2 + \cdots + d_r) + n-1}{n-1} \\
& & -  \binom{\ell}{2} \binom{2 d_2 - d + n-1}{n-1} - \ell (\ell -1)  \binom{d_1 + d_2 - d+ n-1}{n-1}.
\end{eqnarray*}

Moreover, $\sigma_\ell(\X_{n-1,\lambda})$ fills its ambient space if and only if one of the following conditions is satisfied:
\begin{itemize}
\item[(i)] $n = 4$, $\ell = 2$,  and  $\lambda \in \{[1,1], [2,1], [1,1,1]\}$ \quad or 

\item[(ii)] $n = 2 \ell \ge 6$ and   $\lambda = [1,1]$.
\end{itemize}
\end{thm}

 \begin{proof}  Let $P_1, \ldots , P_\ell$ be general points on $\X_{n-1, \lambda}$ and set   
\[
A = \k[x_1,\dots,x_n]/(I_{P_1}+\cdots +I_{P_\ell}). 
\]
Then we have seen in Theorem \ref{HS and dim sigma ell} that
\begin{equation}
    \label{eq:hilb series}
\HS(A) = \frac{1}{(1-t)^n} \left [  1-\sum^r_{i=1} t^{d-d_i} +(r-1)t^d  \right] ^\ell .
\end{equation} 
 
Observing that
\[
2 (d-d_1) + (d-d_2) \ge 2 d_2 + (d_1 + d_3 + \cdots + d_r)  > d,
\]
we get
\begin{eqnarray*}
\left [  1-\sum^r_{i=1} t^{d-d_i} +(r-1)t^d       \right] ^\ell & = &
\left [  1-\sum^r_{i=1} t^{d-d_i} \right] ^\ell + (r-1) \ell \cdot t^d + \cdots \\[1ex]
& = & 1 - \ell \sum^r_{i=1} t^{d-d_i} + \sum_{k=2}^\ell (-1)^k \binom{\ell}{k} \cdot t^{k (d-d_1)} \\
& & \ + \ell (\ell -1) \cdot t^{d-d_1 + d-d_2} + \binom{\ell}{2} \cdot t^{2 (d - d_2)} + (r-1) \ell \cdot t^d + \cdots
\end{eqnarray*}
where only the terms whose exponent of $t$  are potentially at most $d$ have been written out.
Using also 
\[
\frac{1}{(1-t)^n}  = \sum_{j \ge 0} \binom{j+n-1}{j} \cdot t^j , 
\]
 we get from Equation \eqref{eq:hilb series}
\[
\HS(A) = \left [  1-\sum^r_{i=1} t^{d-d_i} +(r-1)t^d       \right] ^\ell \cdot  \left [  \sum_{j \ge 0} \binom{j+n-1}{j} \cdot t^j  \right ] .
\]

The coefficient of $t^d$ in $\HS (A)$ is
\begin{eqnarray*}
 \dim_{\k} [A]_d  & = & \binom{d+n-1}{n-1} - \ell \sum^r_{i=1} \binom{d_i + n-1}{n-1}  + (r-1) \ell \\
& & +  \sum_{k=2}^\ell (-1)^k \binom{\ell}{k} \binom{d_1 - (k-1) (d_2 + \cdots + d_r) + n-1}{n-1} \\[1ex]
& & +  \binom{\ell}{2} \binom{2 d_2 - d + n-1}{n-1} + \ell (\ell -1)  \binom{d_1 + d_2 - d+ n-1}{n-1}.
\end{eqnarray*}

This gives  
\begin{eqnarray*}
 \dim \sigma_{\ell}  (\X_{n-1, \lambda}) & = & -1 + \dim_{\k} [I_{(1)}+\cdots +I_{(\ell)})]_d \\
& &  \\
& = & -1 + \ell \sum^r_{i=1} \binom{d_i + n-1}{n-1} - (r-1) \ell \\
& & \ - \sum_{k=2}^\ell (-1)^k \binom{\ell}{k}  \binom{d_1 - (k-1) (d_2 + \cdots + d_r) + n-1}{n-1} \\
& & -  \binom{\ell}{2} \binom{2 d_2 - d + n-1}{n-1} - \ell (\ell -1)  \binom{d_1 + d_2 - d+ n-1}{n-1} .
\end{eqnarray*}
Then using Formula (\ref{dim X}), we get
\begin{eqnarray*}
 \dim \sigma_{\ell}  (\X_{n-1, \lambda})  & = & \ell \cdot \dim \X_{n-1, \lambda} + (\ell -1)  \\
 & & \ - \sum_{k=2}^\ell (-1)^k \binom{\ell}{k}  \binom{d_1 - (k-1) (d_2 + \cdots + d_r) + n-1}{n-1} \\
 & & -  \binom{\ell}{2} \binom{2 d_2 - d + n-1}{n-1} - \ell (\ell -1) \binom{d_1 + d_2 - d+ n-1}{n-1},
\end{eqnarray*}
as claimed. 

It remains to show the characterization, for $2\ell \leq n$, of when $\sigma_{\ell}  (\X_{n-1, \lambda})$ fills its ambient space. 

This clearly occurs if and only if $[A]_d = 0$. If $2\ell < n$, then $[A]_d$ cannot be zero because $A$ is not artinian as $\dim A = n - 2 \ell$. 

Let $n = 2 \ell$. In this case we can write 
$$
\HS(A) = \left[ \frac{1}{(1-t)^{2 \ell}} \left (1-\sum_{i=1}^r t^{d-d_i} + (r-1)t^d \right ) \right] ^\ell .
$$
Remark \ref{info about S/J} implies that the artinian reduction of $S/J$ is level of socle degree $d-2$.  Hence, each factor 
\[
\frac{1}{(1-t)^{2}} \left (1-\sum_{i=1}^r t^{d-d_i} + (r-1)t^d \right )
\]
is a polynomial of degree $d-2$.  It follows that $\HS (A)$ is a polynomial of degree $\ell (d-2)$. Since $A$ is artinian, this shows that $[A]_d = 0$ if and only if $\ell (d-2) < d$. This is equivalent to
\[
d < 2 + \frac{2}{\ell-1}.
\]
If $\ell = 2$ (hence $n=4$), then we can have $d=2$ or 3 and so $\lambda$ can be $[1,1], [2,1]$ or $[1,1,1]$. If $\ell > 2$ then we must have $d = 2$ and $\lambda = [1,1]$. 
 \end{proof}

\begin{rem}\label{correctionExample}
Note that $\dim \sigma_\ell(\X_{n-1, \lambda})<N-1<\ell(\dim_\k(\X_{n-1, \lambda}))+\ell-1$ can occur (where $N=\binom{d+n-1}{n-1}$), 
as happens, for example, when $n=4$, $\lambda=[2,2]$ and $\ell=2$.
Thus, when $\sigma_\ell(\X_{n-1, \lambda})$ does not fill its ambient space, the defect in Theorem \ref{thm:dim formula} 
need not be given by
$$\textstyle\sum_{k=2}^\ell (-1)^k \binom{\ell}{k}  \binom{d_1 - (k-1) (d_2 + \cdots + d_r) + n-1}{n-1} +  
\binom{\ell}{2} \binom{2 d_2 - d + n-1}{n-1} + \ell (\ell -1)  \binom{d_1 + d_2 - d+ n-1}{n-1}.$$
In general one must use
$$\textstyle\sum_{k=2}^\ell (-1)^k \binom{\ell}{k}  \binom{d_1 - (k-1) (d_2 + \cdots + d_r) + n-1}{n-1} +  
\binom{\ell}{2} \binom{2 d_2 - d + n-1}{n-1} + \ell (\ell -1)  \binom{d_1 + d_2 - d+ n-1}{n-1}-\epsilon,$$
where $\epsilon=\max(0,\ell(\dim_\k(\X_{n-1, \lambda}))+\ell-1-(N-1))$.
\end{rem}

\begin{rem}\label{is zero}
Observe that the last term in the formula of Theorem \ref{thm:dim formula} is zero if and only if $r \ge 3$ and that the penultimate term is zero unless $ r = 2$ and $d_1 = d_2$. 
\end{rem}

\begin{rem}\label{rem:syzygy interpretation}
There is an interesting way to interpret the formula in Theorem \ref{thm:dim formula}.

Let $\mathcal I \subset S$ be an ideal generated by $\ell$ general forms of degree $s = d_2 + \dots + d_r$, where $\ell \leq n$.  Let $\Syz$ be the module of first syzygies of $\mathcal I$, i.e., the sequence 
$$
0 \longrightarrow \Syz \longrightarrow  S(-s)^\ell \longrightarrow {\mathcal I} \longrightarrow 0
$$
is exact.

From the Koszul complex we obtain the following resolution of $\Syz$,
\begin{equation}\label{KoszulES}
0\rightarrow S(-\ell s )^{\binom{\ell}{\ell}} \rightarrow \cdots \rightarrow S(-3 s)^{\binom{\ell}{3}}\rightarrow S(-2 s)^{\binom{\ell}{2}} \rightarrow \Syz\rightarrow 0
\end{equation}\label{SyzFormula}
and so 
\begin{equation}\label{SyzFormula2}
\dim_{\k} [\Syz]_d = \sum_{k=2}^{\ell} (-1)^{k-1} \binom{\ell}{k} \binom{d-k s +n-1}{n-1} .
\end{equation}
\end{rem}

Using Remark \ref{rem:syzygy interpretation} and Theorem \ref{thm:dim formula} we get
\begin{cor}\label{syzygycor}
Let $2\ell \leq n$ and $\lambda = [d_1, \ldots , d_r] \vdash d$.  Let $\mathcal I \subset S$ be an ideal generated by $\ell$ general forms of degree $s=d_2 + \cdots + d_r$ and let $\Syz$ be the first syzygy module of $\mathcal I$.

Then 
\begin{align*}
 \dim \sigma_\ell(\X_{n-1,\lambda}) =  & \ \ell \dim \X_{n-1,\lambda} + (\ell - 1) - \dim_{\k} [\Syz]_d \\[.5ex]
& - \binom{\ell}{2} \binom{2 d_2 - d + n-1}{n-1} - \ell (\ell -1) \binom{d_1 + d_2 - d+ n-1}{n-1} \ .
\end{align*}
\end{cor}

\begin{rem} \label{defect with syz} It is immediate from Corollary \ref{syzygycor} that if $r\geq 3$ then
$$
 \dim \sigma_\ell(\X_{n-1,\lambda})= \ell \dim \X_{n-1,\lambda} + (\ell - 1) - \dim_{\k} [\Syz]_d
$$
and if $r=2$ and $d_1> d_2$ then
$$
 \dim \sigma_\ell(\X_{n-1,\lambda})= \ell \dim \X_{n-1,\lambda} + (\ell - 1) - \dim_{\k} [\Syz]_d - \ell(\ell -1) .
$$
\end{rem}
 
We now discuss the defectivity of $\sigma_\ell(\X_{n-1, \lambda})$ if $2\ell \leq n$. We note that we 
state additional results in the case $\ell = 2$ without this restriction on $n$ in Section \ref{sec:secant lines}.
 
For convenience, we consider the case $r = 2$ separately. 

\begin{thm}\label{general r=2}
Let  $r=2$ with $\ell \le  \frac{n}{2}$ and $\lambda = [d_1,d_2]$.   
Then $\sigma_\ell (\X_{n-1,\lambda})$ 
fills its ambient space if and only if $n = 2 \ell$ and 
$\lambda = [1,1]$, or $n=4$, $\ell=2$ and $\lambda = [2,1]$.

In all other cases,  $\sigma_\ell (\X_{n-1,\lambda})$ is defective and the defect is 
\begin{equation*}
\delta_\ell = \begin{cases}
2 \ell (\ell -1) - \epsilon & \text{if } d_1 = d_2, \text{ and}\\
\ell (\ell -1) - \epsilon+\dim_{\k} [\Syz]_d & \text{if }  d_1 >d_2, 
\end{cases}
\end{equation*}
where an explicit formula for $\dim_{\k} [\Syz]_d$ is given in \eqref{SyzFormula2},
and 
$$\epsilon= \max\{ 0,\ \ell\cdot \dim_\k(\X_{n-1,\lambda}) + \ell -1-(N-1)\}.  $$
 
\end{thm}

\begin{proof} 
Theorem \ref{thm:dim formula} implies the first claim, and that $\sigma_\ell (\X_{n-1,\lambda})$ does not fill its ambient space
in all other cases. Keeping Remark \ref{correctionExample} in mind, we get $\delta_\ell$ for the case that $d_1 = d_2$ from Theorem \ref{thm:dim formula},
and from Remark \ref{defect with syz} for the other case.
\end{proof}

\begin{rem}
The paper \cite{CCG:1} solves the problem of determining when $\sigma_\ell(\X_{n-1,\lambda})$ fills its ambient space in the case when $r=2$ and $2 \ell \leq n+1$, in quite different language. We should point out, first, that their $r$ is our $\ell$ and their $n$ is our $n-1$. Most of the cases that they need to consider satisfy $2 \ell = n+1$, and so do not overlap with Theorem \ref{general r=2} (but see Theorem \ref{thm:dim when froeberg} (a)).  Nevertheless, the second sentence of Theorem \ref{general r=2} follows from Theorem 5.1 of \cite{CCG:1}.
We included it above for reference and because it is such an easy consequence of our current approach.
\end{rem}

\begin{thm} \label{general r geq 3}
Let $\lambda \vdash d$, $\lambda =[d_1, \ldots , d_r]$, where $r \ge 3$.   Assume $2\ell \le n$. 
\begin{enumerate} 

\item[(a)] If $d_1 < d_2 + \dots + d_r$, then 
  $$ 
\dim \sigma_\ell( \X_{n-1,\lambda} ) = \ell \cdot \dim \X_{n-1,\lambda}+ (\ell - 1) \leq N-1
$$
 and $\sigma_\ell( \X_{n-1,\lambda})$ is not defective.

\item[(b)] If $d_1 \geq d_2 + \dots + d_r$, then $\sigma_\ell (\X_{n-1,\lambda})$ is defective, with defect
$$
\delta_\ell = \dim_{\k} [\Syz]_d -\epsilon, 
$$
where $\dim_{\k} [\Syz]_d$ is given explicitly in \eqref{SyzFormula}
and $\epsilon$ is as given in Remark \ref{correctionExample}.
\end{enumerate}
\end{thm}

\begin{proof}
We begin with (a). 
The proof is immediate from Theorem \ref{thm:dim formula} since, as partly noted
in Remark \ref{is zero}, all the summands in that formula vanish except for $\ell \cdot \dim \X_{n-1,\lambda}+ (\ell - 1)$.

As for (b), Theorem \ref{thm:dim formula}  shows that $\sigma_\ell( \X_{n-1,\lambda})$ does not fill its ambient space. 
Thus, if $N-1\leq\ell(\dim \X_{n-1,\lambda})+\ell-1$, then $\sigma_\ell( \X_{n-1,\lambda})$ is defective, and using
Remarks \ref{correctionExample} and \ref{defect with syz} we see the defect is 
$\delta_\ell=\dim_{\k} [\Syz]_d-\epsilon$. 
Suppose now that
$N-1>\ell(\dim \X_{n-1,\lambda})+\ell-1$, so $\epsilon=0$
Hence, Remark \ref{defect with syz} gives that the defect is
\[
\delta_\ell = \dim_{\k} [\Syz]_d=\dim_{\k} [\Syz]_d-\epsilon.
\]
It remains to show that it is positive. However, $\Syz$ is  a submodule of a free module, and the generators of $\Syz$ have degree $2(d_2+\cdots+d_r) \le d$ (see \eqref{KoszulES}). We conclude that $\dim_{\k} [\Syz]_d >0$, and so $\sigma_\ell(\X_{n-1,\lambda})$ again is defective.
\end{proof} 

\begin{rem} \label{PDH} 
Theorems \ref{general r=2} and \ref{general r geq 3} give
us our first view of the fact that the hyperplane $d_1 = d_2 + \dots + d_r$ in $\mathbb N^r$ separates two very different kinds of behaviors with respect to defectivity, when $[d_1, \ldots , d_r] = \lambda\vdash d$ is a partition of $d$ into $r\geq 2$ parts. This was observed for $n=2$ in \cite{CGGS}, and it will recur frequently in this paper. As a result, we will follow \cite{CGGS} in referring to this as the {\em partition dividing hyperplane in $\mathbb N^r$}.
\end{rem}

\begin{rem} \label{sandro rmk}  
The formula for $\dim \sigma_\ell (\X_{n-1,\lambda})$ given in Conjecture \ref{conj:main conj}(b) and Theorem \ref{thm:dim formula} can be interpreted intuitively by using Terracini's Lemma via Corollary \ref{dim sigma ell} and the identity $\dim_{\k} [I_{P_1}+ \cdots +I_{P_\ell}]_d = \dim_{\k}([I_{P_1}]_d+ \cdots +[I_{P_\ell}]_d)$. 
The simplest case occurs when the spaces $[I_{P_i}]_d$ meet pair-wise in just $0$. In that case
$\dim_{\k}([I_{P_1}]_d+ \cdots +[I_{P_\ell}]_d)=\dim_{\k}[I_{P_1}]_d+ \cdots +\dim_{\k} [I_{P_\ell}]_d
=\ell(\dim_{\k} [I_{P_1}])=\ell(1+\dim \X_{n-1,\lambda})$ so 
$\dim \sigma_\ell (\X_{n-1,\lambda})=\ell(\dim \X_{n-1,\lambda})+\ell-1$.

Often the pairs will not meet only in 0. To consider that case,
we set $\underline{i}=\{i_1,\ldots,i_j\}$ for $1\leq i_1<\cdots< i_j\leq \ell$ and say $|\underline{i}|=j$.
We then define $V_{\underline{i}}=\cap_{t\in\underline{i}} [I_{P_t}]$. 
For $1\leq u\leq\ell$, let $v_u=\sum_{|\underline{i}|=u}\dim V_{\underline{i}}$,
so $v_1=\dim_{\k}[I_{P_1}]_d+ \cdots +\dim_{\k} [I_{P_\ell}]_d$,
$v_2$ is the sum of the dimensions of the pair-wise intersections of the $[I_{P_i}]_d$,
$v_3$ is the sum of the triple intersections, and so on.
Inclusion-exclusion now gives
$\dim_{\k}([I_{P_1}]_d+ \cdots +[I_{P_\ell}]_d)=\sum_{1\leq u\leq\ell}(-1)^{u+1}v_u$.

Let's look at this in the case that  $\lambda$ is such that $d_1\geq s = d_2+d_3+ \cdots +d_r$; i.e.  $\lambda$ is 
above the ``partition dividing hyperplane" in $\mathbb{N}^r$.
This is an example for which it is not possible that 
$[I_{P_{i_1}}]_d\cap[I_{P_{i_2}}]_d=0$ for $i_1\neq i_2$.
To see this, say each $P_j$ corresponds to the form
$F_{j,1}F_{j,2} \cdots F_{j,r}$.
Then $I_{P_i} \cap I_{P_j}$ will, for every $i,j\in \{1, \dots ,\ell\}$, $i\neq j$, contain all the products of the type: 
$F_{i,2} \cdots F_{i,r}F_{j,2} \cdots F_{j,r}G$,  
where $G$ is any form of degree $d-2s=d_1-s$. 
Thus $\dim_\k [I_{P_{i_1}}]_d\cap[I_{P_{i_2}}]_d\geq \binom{d-2s+n-1}{n-1}$ for each
pair $i_1\neq i_2$, so $v_2\geq \binom{\ell}{2}\binom{d-2s+n-1}{n-1}$.
Moreover, when $r=2$, $I_{P_i}\cap I_{P_j}$ will also contain  the forms 
$F_{1,i}F_{2,j}$, 
for all $i\neq j\in \{1, \ldots ,\ell \}$ so in this case $v_2\geq \binom{\ell}{2}\binom{d-2s+n-1}{n-1}+\ell(\ell-1)$. 
If $r=2$ and $d_1=d_2$, we also have 
$F_{1,i}F_{1,j}$ in the intersection, so $v_2\geq \binom{\ell}{2}\binom{d-2s+n-1}{n-1}+\ell(\ell-1)+\binom{\ell}{2}$. 

Similarly, if $d_1\geq 2s$, then $\dim_\k [I_{P_{i_1}}\cap I_{P_{i_2}}\cap I_{P_{i_3}}]_d\geq \binom{d-3s+n-1}{n-1}$, 
since 
$$F_{i_1,2} \cdots F_{i_1,r}F_{i_2,2} \cdots F_{i_2,r}F_{i_3,2} \cdots F_{i_3,r}H\in [I_{P_{i_1}}\cap I_{P_{i_2}}\cap I_{P_{i_3}}]_d$$ 
for any $H\in S_{d-3s}$. Hence $v_3\geq\binom{\ell}{3}\binom{d-3s+n-1}{n-1}$.
In the same way, if $d_1\geq ks$, then $v_k\geq\binom{\ell}{k}\binom{d-ks+n-1}{n-1}$.

If the lower bounds on each $v_u$ above were to equal the corresponding $v_u$, then 
$\dim_{\k}([I_{P_1}]_d+ \cdots +[I_{P_\ell}]_d)=\sum_{1\leq u\leq\ell}(-1)^{u+1}v_u$
becomes precisely the formula given in 
Conjecture \ref{conj:main conj}(b) and proved in a special case in Theorem \ref{thm:dim formula}.
Of course, because cancellations may occur in an alternating sum, it is possible
for the formula in Conjecture \ref{conj:main conj}(b) to hold even if some of the lower bounds
were to be strictly less than their corresponding $v_u$. However, this paper makes the case that the proper context for Conjecture \ref{conj:main conj} is intersection theory of projective schemes. 
When the intersections are not proper, we will show in the coming sections that Conjecture \ref{conj:main conj} is a consequence of conjecturing that a Lefschetz property governs the behavior of the intersections. This presents a novel and structural approach to the dimension of secant varieties that we feel makes the conjecture much more strongly motivated than it otherwise would be (since it gives a unified perspective for the case of proper intersections, where we prove the formula of Conjecture \ref{conj:main conj}, and the case of improper intersections, where the formula is only conjectural), and furthermore provides new tools for the computations. 
\end{rem}


\section{The Secant Line Variety and Passage to Improper Intersections} \label{secant line section} 
     \label{sec:secant lines}

A consequence of the work done up to this point is that if $n \geq 4$ we have the following result for the secant line variety of the variety of reducible forms of prescribed type:

\begin{thm}
     \label{thm: l = 2, proper intersection} 
Assume $\ell = 2$ and $n \geq 4$. Then, for all partitions $\lambda = [d_1,\dots,d_r]$ of $d$  with $r \geq 2$: 
\begin{itemize}

\item[(a)]  \begin{eqnarray*}
  \dim \sigma_{2}  (\X_{n-1, \lambda})  & = &  2 \dim \X_{n-1, \lambda} + 1 - \binom{d_1 -  (d_2 + \cdots + d_r) + n-1}{n-1}  \\[.5ex]
&& \  - \binom{2 d_2 - d + n-1}{n-1} - 2 \binom{d_1 + d_2 - d+ n-1}{n-1};
\end{eqnarray*} 

\item[(b)]  $\sigma_2( \X_{n-1,\lambda})$ fills its ambient space if and only if $\lambda \in \{[1,1], [2, 1], [1,1,1]\}$ and $n = 4$;

\item[(c)] if $\sigma_2( \X_{n-1,\lambda})$ does not fill its ambient space, then 
$\sigma_{2}  (\X_{n-1, \lambda})$ is not defective if and only if $d_1 < d_2 + \cdots + d_r$; and 

\item[(d)] if $\sigma_2( \X_{n-1,\lambda})$ is defective, then the defect is 
\[
\binom{d_1 -  (d_2 + \cdots + d_r) + n-1}{n-1}  + \binom{2 d_2 - d + n-1}{n-1} + 2 \binom{d_1 + d_2 - d+ n-1}{n-1}-\epsilon, 
\]
where $\epsilon$ is as given in Remark \ref{correctionExample}.
\end{itemize}
\end{thm}

\begin{proof}
Parts (a) and (b) are immediate from Theorem \ref{thm:dim formula}. 
Theorem \ref{thm:defectivity - intro} gives part (c). Claim (d) is a consequence of (a). 
\end{proof} 

Notice that Theorem \ref{thm: l = 2, proper intersection} assumes $n\geq 4$, since $\ell=2$ but it relies on results that assume $n\geq 2\ell$.  The main purpose of this section is to understand what is needed to pass beyond the condition $2 \ell \leq n$ with our approach. We begin with a review of the main results of \cite{CGGS}, since that paper gives a careful analysis of the case $\ell = 2$, $n = 3$ using entirely different methods. Then we will see what would be needed in order to pass from the results of the previous section to this case. Having the essential idea in hand, subsequent sections of this paper will carry out the calculations.  The main result is a single, explicit (albeit complicated) conjectured formula for the dimension of the secant variety for {\em any} choice of $n$, $\lambda$ and $\ell$ (see Conjecture \ref{conj:main conj}  and Theorem \ref{thm: gen dim formula}). In particular, the results of the previous sections agree with this conjecture.  We are then able to give many consequences, some proven unconditionally, some conjectural.

So first we recall the results of \cite{CGGS}, which in particular  confirm the importance of the ``partition dividing hyperplane" mentioned in Remark \ref{PDH}. 
 We will see that, in general, the partitions \lq\lq below" the hyperplane (i.e., those partitions for which $d_1 < d_2+\cdots+d_r$) behave quite differently from those \lq\lq above" the hyperplane (i.e., those partitions for which $d_1 \geq d_2+\cdots+d_r$). In fact, we expect this to be true in all cases (see Conjecture \ref{conj:defectivity conj} and Proposition \ref{prop:non-defective}).

Propositions \ref{r>6} and \ref{r less than 6} can be deduced easily from the results of \cite{CGGS}.
We see that the condition for defectivity when $r\geq6$ is straightforward, but there are exceptions 
when $2\leq r<6$.

\begin{prop}\label{r>6} Let $n=3$ with $\lambda= [d_1, \ldots , d_r]\vdash d$ a partition of $d$ into $r \geq 6$ parts
and $s=d_2+\cdots+d_r$.  Set 
$$
p=  \sum _{2\leq i<j\leq r}d_id_j  .
$$
\begin{itemize}
\item[(a)] If $d_1 \geq s$ , then $\sigma_2( \X_{2,\lambda})$ is always defective, and the defect is 
\begin{equation}
      \label{eq:defect} 
\min\left\{ \binom{d_1-s+2}{2}, 2p-3s \right\}.  
\end{equation}

\item[(b)] If $d_1 < s$, then 
$$
\dim \left( \sigma_2( \X_{2,\lambda})\right) = 2\dim ( \X_{2,\lambda}) + 1, 
$$
and hence $\sigma_2( \X_{2,\lambda})$ is not defective.
\end{itemize}
\end{prop}

\begin{prop} \label{r less than 6}
Let $n=3$ with $\lambda= [d_1, \ldots , d_r]\vdash d$ a partition of $d$ into $2\leq r<6$ parts and let $s=d_2+\cdots+d_r$.
\begin{itemize}
  \item[(a)] If $r=2$, then the secant line variety fills its ambient space, and so it is never defective.
\item[(b)]  For the following partitions the secant variety $\sigma_2( \X_{2,\lambda})$ fills its ambient  space and so the defect is zero:
\begin{itemize}

\item[$\bullet$]  $r=3$ and $\lambda \in \{ [d_1,d_2,1], [d_1,2,2], [d_1,3,2], [d_1,4,2], [d_1,5,2],[d_1,6,2], [d_1,3,3] \}$;

\item[$\bullet$]   $r=4$ and  $\lambda \in \{ [d_1,1,1,1],[d_1,2,1,1],[d_1,3,1,1],[d_1,4,1,1] \}$;  

\item[$\bullet$]  $r = 5$ and $\lambda = [d_1,1,1,1,1]$.  
  
 \end{itemize}

 \item[(c)] Apart from the partitions described above, if $d_1\geq s$ and $r \ge 3$, then $\sigma_2( \X_{2,\lambda})$ is always defective.  In this case the defect is equal to \eqref{eq:defect}  above.

\item [(d)] If $d_1<s$ then $\sigma_2( \X_{2,\lambda})$ is never defective.  Apart from the partitions described in (b),  the secant line variety has dimension $2 \dim( \X_{2,\lambda}) + 1$.  
\end{itemize} 
\end{prop}

\begin{ex}  Consider $\lambda = [2,2,2,1]$.  This partition has $d_1<s$ and so we are below the partition dividing hyperplane.
By Proposition \ref{r less than 6}(d), $\dim \sigma_2( \X_{2,\lambda})=2\dim \X_{2,\lambda}+ 1$ and a computation shows that
$2\dim \X_{2,\lambda}+ 1=N-1$, hence $\sigma_2( \X_{2,\lambda})$ fills its ambient space for this example.
\end{ex}

\begin{rem} \label{hyper surf} In case $\ell = 2$, $\lambda\vdash d, \ \lambda =  [d_1, \ldots , d_r]$, with $d_1 = s$ we can show that the only time that $\sigma_2 ( \X_{2,\lambda})$ is a hypersurface in its ambient space is when $\lambda = [9,7,2], [5,2,2,1]$ and $[7,5,1,1]$.  It would be interesting to find the equations of these hypersurfaces.  These are all defective secant line varieties, and up to this point whenever we have been able to find equations for such defective secant varieties they have been determinants.  Is that the situation in this case as well?   
\end{rem}

Now fix $n = 3$ and $\ell = 2$. Then the codimension of $\sigma_{2} (\X_{2, \lambda})$ is 
given by \linebreak $\dim_{\k} [S/(I_{P_1}+I_{P_2})]_d$ (see 
Corollary \ref{dim sigma ell}). By Lemma \ref{sop} and Theorem \ref{thm:join as tensor product}, the latter is equal to $\dim_{\k} [B/ \mathcal L B]_d$, where $B = T/(I_{(1)}^{(e)}+I_{(2)}^{(e)})$ and $\mathcal L \subset T$ is a regular sequence comprised of $(\ell-1)n = 3$ general linear forms.  By Remark \ref{formula}, we know the Hilbert function of $B$.
Thus, the important question now is to determine the relation between the Hilbert function of $B$ and that of $B/ \mathcal L B$. 

Since $\dim B = \ell(n-2) = 2$ (Theorem \ref{thm:join as tensor product}), the first two linear forms in $\mathcal L$ form a $B$-regular sequence by Proposition \ref{prop:proper intersection}. Say $\mathcal L = \{L_1,L_2,L_3 \}$ and $L_1,L_2$ form a regular sequence. Thus we have to find the Hilbert function after reducing by one more general linear form,  $L_3$.  Let $\mathcal L' = \{L_1,L_2 \}$ and let $ B' = B/\mathcal L' B$. 

Note that there is an exact sequence
\begin{equation} \label{mult SES}
 B'(-1) \stackrel{\times L_3}{\longrightarrow}  B' \rightarrow  B'/ L_3  B' \rightarrow 0.
\end{equation}
Thus,  in order to determine the Hilbert function of $B/\mathcal L B \cong  B'/ L_3   B'$ it is enough to know that $\times L_3$ has maximal rank at each degree.  This is exactly what is provided by the Weak Lefschetz Property, as described in the next section. It is the basis for the remaining calculations and the general formulae that they give.

As an example, let us consider a case mentioned in Remark \ref{hyper surf} and verify that  if $\lambda = [9,7,2]$, and if the maximal rank property holds for $ B'$, then our results of the previous sections imply that $\sigma_2 (\X_{2,\lambda})$ is a hypersurface.    Using Remark \ref{formula}, we get for the Hilbert series of $ B'$
\begin{align*}
\HS ( B') &  = (1 - t)^2 \cdot \HS (B)  \\[.5ex] 
& = \frac{(1-t)^6}{(1-t)^4}  \cdot \HS(B) \\[.5ex]
& =  \frac{1}{(1-t)^4}\left [1 - t^{16} - t^{11} - t^9 + t^{18} \right ]^2 \\[.5ex] 
& = 1 + 4 t + \cdots + 634 t^{17} + 635 t^{18} + \cdots + 4 t^{32}.  
\end{align*}
Thus, 
\[ 
h_{17} = \dim_{\k} [ B']_{17} = 634 \quad \text{ and } \quad h_{18} = \dim_{\k} [ B']_{18} = 635 .
\] 
Note that $d = 18$ in our example. If multiplication by $L_3$ on $ B'$ has maximal rank, we obtain (see Sequence  \eqref{mult SES}) that for $S = \k[x_1,x_2,x_3]$, 
\[
\dim_{\k} [S/(I_{P_1}+I_{P_2})]_d = \dim_{\k} [ B'/L_3  B']_d =  \max \{ h_d - h_{d-1}, 0 \} = 1.
\]
Since this is precisely the  codimension of $\sigma_2 (\X_{2,\lambda})$, we see that under the hypothesis that multiplication by $L_3$ has maximal rank we  obtain that indeed $\sigma_2 (\X_{2,\lambda})$ is a hypersurface in its ambient space.

For arbitrary $n$ and $\ell > \frac{n}{2}$,  we will need that the maximal rank property holds sequentially, using the right number of linear forms, until we arrive at $n$ variables.  In the following sections we make use of this idea to give a general formula for the dimensions of the secant varieties (see Theorem \ref{thm: gen dim formula} and Conjecture \ref{conj:main conj}), and as a consequence we describe the defective cases, assuming that suitable maximal rank properties hold. In many cases we know for different reasons that these maximal rank properties do hold, and in those cases we obtain unconditional (not conjectural) formulas. 
Given the many special cases and the seemingly disparate results covering them that have been found up to now, we were astonished to find a simple unifying principle that produces a single conjectural formula for the exact dimension of the secant variety for any given $n$, $\lambda$ and $\ell$.


\section{Improper Intersections and Lefschetz properties}  
\label{sec:improper inters}

We now consider the case in which the $\ell$ varieties $V(I_{(j)})$ ($1 \le j \le \ell$) 
that determine tangent spaces at $\ell$ general points to $\X_{n-1, \lambda}$ intersect improperly. By Proposition \ref{prop:proper intersection}(b), their intersection is proper if $2\ell \le n$. However, if $2\ell > n$ then we will see that  the intersection is improper,  so in particular the diagonal trick from intersection theory becomes more difficult to apply.  To make up for this, we will   use Lefschetz properties as formally introduced in \cite{HMNW} in order to determine $\dim_{\k} [S/(I_{(1)} + \cdots + I_{(\ell)})]_d$. Because we are dealing with general forms, these properties are known in some cases, and we conjecture them in the remaining cases.

More precisely, we conjecture that a general artinian reduction of the coordinate ring of the join of $V(I_{(1)}),\ldots,V(I_{(\ell)})$ has enough Lefschetz elements (see Conjecture \ref{conj:WLP} below). 

We continue to use the notation introduced above. In particular, 
\[
B = S/I_{(1)} \otimes_{\k} \cdots \otimes_{\k} S/I_{(\ell)}  \cong T/\tilde{I}
\] 
is the coordinate ring of the join of $V(I_{(1)}^{(e)}),\ldots,V(I_{(\ell)}^{(e)})$ in $\PP^{\ell n - 1}$. 

As above, let $\cL$ be a set of $(\ell -1) n$ general linear forms in $T$. Let $\cL' \subset \cL$ be a subset consisting of $\min \{\ell (n-2), (\ell - 1) n\}$ such forms. Thus, 
\[
\cL' = \begin{cases}
\cL & \text{if } \ell \le \frac{n}{2}\\
\subsetneqq \cL & \text{if } \ell > \frac{n}{2}
\end{cases}.
\]
Then there is the following useful observation. 

\begin{prop}
    \label{prop:improper inters}
If  $2\ell \ge n$, then: 
\begin{itemize}

\item[(a)] The linear forms in $\cL'$ form a $B$-regular sequence and $\dim B/\cL' B = 0$. 

\item[(b)] $\codim (I_{(1)}+\cdots+I_{(\ell)})= n$, and hence the varieties $V(I_{(1)}),\ldots,V(I_{(\ell)})$ intersect improperly if $2\ell > n$. 

\end{itemize}
\end{prop}

\begin{proof}
By Theorem \ref{thm:join as tensor product}, $B$ is Cohen-Macaulay of dimension $\ell (n-2)$. Hence, Claim (a) follows by the generality of the linear forms in $\cL'$.

Part (a) shows in particular that $\dim B/\cL B = 0$. Hence Lemma \ref{sop} gives 
\[
\codim (I_{(1)}+\cdots+I_{(\ell)})= n, 
\]
and we are done. 
\end{proof} 

By Lemma \ref{sop}, we are interested in the Hilbert function of $B/\cL B$. We know the Hilbert function of $B$, and thus the Hilbert function of its general artinian reduction $B/\cL' B$ by Proposition \ref{prop:improper inters}(a). However, if $\ell > \frac{n}{2}$, then $\cL' \neq \cL$. 

Recall that a linear form $L$ is a non-zerodivisor of a graded algebra $A$ if the multiplication by $L$ on $A$ is injective. If $A \neq 0$ is artinian, this cannot be true. However, one may hope that this multiplication still has maximal rank. This has been codified in \cite{HMNW} as follows:

\begin{defn}
   \label{def:WLP} 
Let $A = S/J$ be an artinian graded $\k$-algebra. Then $A$ is said to have the  {\em Weak Lefschetz Property} if, for each integer $i$, multiplication
by a general linear form $L \in A$ from $[A]_i$ to $[A]_{i+1}$ has maximal
rank. In this case, the form $L$  is called a {\em Lefschetz element} of $A$. 

We say $A$  has the {\em Strong Lefschetz Property} if, for all $i$ and $e$, multiplication by $L^e$ from $[A]_i$ to $[A]_{i+e}$ has maximal rank.  
\end{defn}

The first systematic study of these properties in this generality was carried out in \cite{HMNW}.  In particular,  the Hilbert functions of algebras with either 
the \WLP\ or \SLP\ were classified there, and a sharp bound was given on the possible graded  Betti numbers. Thus, the presence of these
properties leads to strong restrictions on the possible invariants. Many natural families of algebras are expected to have  a Lefschetz property. However, it is typically very difficult to establish this. 
We refer to \cite{HMMNWW} and \cite{wlp-survey} for  further information and results. 

There is a useful numerical characterization of Lefschetz elements. To state it we need some notation. 

\begin{defn}
    \label{def:positive part series} 
Let $\sum_{i \ge 0}  a_i t^i$  be a formal power series, where $a_i \in \ZZ$. Then we define an associated power series with non-negative coefficients by 
\[
\left | \sum_{i \ge 0}  a_i t^i \right |^+   = \sum_{i \ge 0} b_i t^i,
\]
where 
\[
b_i = \begin{cases}
a_i & \text{if } a_j > 0 \text{ for all }  j \le i \\
0 & \text{otherwise}. 
\end{cases} 
\]
\end{defn}

\begin{lem}
    \label{lem:char Lefschetz element}
Let $A$ be a standard artinian graded algebra, and let $L \in A$ be a linear form. Then the following conditions are equivalent: 
\begin{itemize}

\item[(a)] $L$ is a Lefschetz element of $A$. 

\item[(b)] The Hilbert function of $A/LA$ is given by
\[
\dim_{\k}  [A/L A]_i = \max \{ 0,\; \dim_{\k} [A]_i - \dim_{\k} [A]_{i-1} \} \quad \text{ for all integers } i. 
\]

\item[(c)]  The Hilbert series of $A/LA$ is 
\[
\HS (A/LA) = \left | (1-t) \cdot \HS (A) \right |^+.  
\] 
\end{itemize}
\end{lem} 

This is immediate from the definitions. 

It is natural to ask if an algebra has the Weak Lefschetz Property repeatedly, using more than one linear form.  This notion was first introduced by Iarrobino (according to the introduction of \cite{HW}) and formalized in \cite{HW} and \cite{Constantinescu}.

\begin{defn}
    \label{def:k-WLP}
An artinian standard graded $\k$-algebra $A $ is said to have the {\em $k$-Weak Lefschetz Property} (denoted $k$-WLP) if either $k=0$, or $k>0$ and there are linear forms $L_1,\ldots,L_k \in A$ such that $L_i$ is a Lefschetz element of $A/(L_1,\ldots,L_{i-1}) A$ for all $i = 1,\ldots,k$.     In this case, $\{L_1,\ldots,L_k\}$ is called a {\em $k$-Lefschetz set} of $A$. 
\end{defn} 

By definition, every artinian algebra has the $0$-WLP. Moreover, if an algebra $A$ has the $k$-WLP, then a set of $k$ general linear forms is a $k$-Lefschetz set of $A$. Since all quotients of polynomial rings of at most two variables have the \WLP\ by \cite{HMNW}, the conditions $(n-2)$-WLP and $n$-WLP are equivalent for quotients of $S$. 

Using Lemma \ref{lem:char Lefschetz element}, the above property can be restated as follows. 

\begin{lem}
    \label{lem:char k-WLP}
An artinian standard graded algebra $A$ has the $k$-WLP  if and only if there are linear forms $L_1,\ldots,L_k \in A$ such that 
\[
\HS (A/(L_1,\ldots,L_i) A) =  \left | (1-t)^i \cdot \HS (A) \right |^+ \quad \text{ for all } i \le k.
\] 
\end{lem}

\begin{proof} 
This follows by combining Lemma \ref{lem:char Lefschetz element} and \cite[Lemma 4]{froeberg}. 
\end{proof}

We now want to use these concepts to determine the dimension of  secant varieties $\sigma_{\ell}  (\X_{n-1, \lambda})$, using $B/\cL B$. Since the linear forms in $\cL$ are general, it is reasonable to ask if, in the case where $\ell > \frac{n}{2}$,  the linear forms in $\cL \setminus \cL'$ form a Lefschetz set of the general artinian reduction $B/\cL' B$. We illustrate the usefulness of this property.

\begin{ex}
    \label{exa:WLP}
Consider the case $n = 4$, $\ell = 3$, and $\lambda = [3,2,2]$, so $d = 7$. 
Then Remark \ref{formula} gives
\[
\HS (B) = \frac{(1 - t^4 - 2 t^5 + 2t^7)^3}{(1-t)^{12}}. 
\]
Since $|\cL'| = 6$, Proposition \ref{prop:proper intersection} implies 
\begin{align*}    
\HS (B/\cL' B) & = \frac{(1 - t^4 - 2 t^5 + 2t^7)^3}{(1-t)^6} \\
& =  8 t^{15}+48 t^{14}+144 t^{13}+292 t^{12}+456 t^{11}+588  t^{10}+646 t^9+612 t^8 \\
& \quad +504 t^7+363 t^6+228 t^5+123 t^4+56 t^3+21 t^2+6 t+1\\ 
& = \sum_{i = 0}^{15} c_i t^i
\end{align*}

 Note that the passage from $B$ to $C= B/\cL' B$ corresponds to intersecting the join of $V(I_{P_1}), V(I_{P_2})$, and $V(I_{P_3})$ properly with a linear subspace of codimension three. However, $C$ is artinian, and thus any further hyperplane sections correspond to improper intersections. 

We have checked by computer that $C$ has the $2$-WLP, so let   $L_1, L_2$ be a $2$-Lefschetz set. Then, we may assume that $\cL = \cL' \cup \{L_1, L_2\}$, and Lemma \ref{sop} gives 
\[
\HS (A) = \HS (B/\cL B) = \HS (C/(L_1, L_2) C) = \left | (1-t)^2 \cdot \HS (C) \right |^+.
\]
We compute this in two steps. First, we get 
\begin{align*}   
\HS (C/L_1 C) & = \sum_{i \ge 0} \max \{0, c_i - c_{i-1} \} t^i \\
& = 34 t^9+108 t^8+141 t^7+135 t^6+105 t^5+67 t^4+35 t^3+15 t^2+5 t+1 \\
& = \sum_{i = 0}^{9} b_i t^i. 
\end{align*} 
Thus, we obtain 
\begin{align*}   
\HS (A) =  \HS (C/(L_1, L_2) C) & =  \sum_{i \ge 0} \max \{0, b_i - b_{i-1} \} t^i \\
& = 6 t^7+30 t^6+38 t^5+32 t^4+20 t^3+10 t^2+4 t+1. 
\end{align*}
In particular, the secant variety $\sigma_{3} (\X_{3, \lambda})$ has codimension $\dim_{\k} [A]_d = 6$ in its ambient space. Hence, $\sigma_{3} (\X_{3, \lambda})$ is non-defective of dimension 113     and does not fill its ambient space. 
\end{ex}

Computer experiments (see Remark~\ref{rem:evidence}) suggest that a similar analysis can always be carried out. Thus, we conjecture: 

\begin{conj}[WLP-Conjecture]
    \label{conj:WLP}
The algebra $B/\cL' B$ has the $k$-WLP for  $k = \max \{0, \, 2 \ell - n\}$.   
\end{conj}  

\begin{rem}
    \label{rem:evidence} 
(i) From the point of view of dimensions of secant varieties, we are interested in
computing the value in degree $d$ of the Hilbert function of $S/(I_{(1)} + \cdots + I_{(\ell)}) \cong B/{\La}B$. 
In fact,  we know the entire Hilbert function of $B/\cL' B$ precisely, but this is not enough, since
in general $\cL' \subsetneq \cL$. The new idea, illustrated in Example \ref{exa:WLP}, is that if $B/\cL' B$ has
the $k$-WLP for $k = \max\{ 0, 2 \ell - n\}$,  then we also know the entire Hilbert function of $B/\cL B$ 
by Lemma \ref{lem:char k-WLP}, giving us much more than we need. This idea is carefully worked out in the remainder of this section. 
In particular,  we derive the precise formula of
Theorem \ref{thm:explain a-j} and apply it in Theorem \ref{thm: main thm - intro} to determine the dimension of the secant variety,
our main interest.
    
(ii)     
Conjecture \ref{conj:WLP} is supported by a great deal of computer evidence using CoCoA \cite{cocoa}  and
Macaulay2 \cite{Macaulay2}.  These calculations were performed over $\ZZ/31991 \ZZ$ by using random linear forms, which implies that $k$-WLP holds also in characteristic zero

Checking the conjecture requires the computation of $k+1$ Hilbert functions (see Lemma \ref{lem:char k-WLP}). In particular, the conjecture predicts the Hilbert function of $B/\cL B$. We confirmed this prediction in the following cases: 
\[
\begin{array}{c|c|c|c}
n & d & \ell & \lambda \\ \hline
3 & 4,\dots,10 & 2, \dots 10 & \hbox{all possible partitions} \\
4 & 4,\dots, 10 & 3, \dots, 9 & \hbox{all possible partitions} \\
5 & 4, \dots, 10 & 3,4,5 & \hbox{all possible partitions} \\
6 & 4,\dots, 9 & 4 & \hbox{all possible partitions} 
\end{array}
\]
This gives some evidence for Conjecture \ref{conj:WLP}. 

We give even stronger evidence by confirming Conjecture \ref{conj:WLP} itself in a range of cases.  Specifically, we confirmed the $k$-WLP for $k = \max\{ 0, 2 \ell - n\}$ in the following cases: 
\[
\begin{array}{c|c|c|l}
n & d & \ell & \lambda \\ \hline
3 & 3,\dots,10 & 2& \hbox{all possible partitions} \\
3 & 3,\dots, 8 & 3 & \hbox{all possible partitions}  \\
3 & 3,4, 5 & 4 & \hbox{all possible partitions} \\
4 & 3,\dots, 8 & 3  & \hbox{all possible partitions} \\
\end{array}
\]
\end{rem} 

Notice that, by definition, the WLP-Conjecture \ref{conj:WLP} is true if  $\ell \le \frac{n}{2}$. 

In Section \ref{sec:r is 2} we will see that  in the case  $r = 2$ this conjecture is closely related to a well-known conjecture by Fr\"oberg, lending additional evidence to the WLP-Conjecture. 

 Here we show that if true, the WLP-Conjecture allows us to extend Theorem  \ref{thm:dim formula} to $\ell$ with $2\ell > n$.  In order to express this we need more notation. 

\begin{defn}
   \label{def:integers a-j}
Let  $\lambda = [d_1,\ldots,d_r]\vdash d$  be a partition with $r \ge 2$ , and let $\ell$ and $n$ be positive integers. For $j = 0,\ldots,d$, define integers $a_j = a_j (\ell, n, \lambda)$  by 
\begin{align*}
a_j  = &  \binom{j+n-1}{n-1} - \ell \sum^r_{i=1} \binom{j + d_i - d+ n-1}{n-1}  + (r-1) \ell \binom{j}{d} \\
&  +  \sum_{k=2}^\ell (-1)^k \binom{\ell}{k} \binom{j - k (d - d_1)+ n-1}{n-1} \\[1ex]
& +  \binom{\ell}{2} \binom{j + 2 d_2 -  2 d + n-1}{n-1} + \ell (\ell -1)  \binom{j + d_1 + d_2 - 2 d+ n-1}{n-1}. 
\end{align*}   
\end{defn}

Observe that $a_j (\ell, n, \lambda) > 0$ if $0 \le j < s = d_2 + \cdots + d_r$ as, for example, $\binom{j + d_i - d+ n-1}{n-1} = 0$ in this case.  

Now we explain the meaning of the numbers $a_j(\ell,n,\lambda)$.

\begin{thm}
    \label{thm:explain a-j}
Assume that the WLP-Conjecture is true for some $\ell, n$, and $\lambda$.   Let $P_1, \ldots , P_\ell$ be general points on $\X_{n-1, \lambda}$,  and set   $A = S/(I_{P_1}+\cdots +I_{P_\ell})$.
If $i \le d$ is a non-negative integer
then 
\[
\dim_{\k} [A]_i = \begin{cases}
0 & \text{if $a_j \le 0$ for some  $j$ with $0 \le j \le i$} \\
a_i > 0 & \text{otherwise}.  
\end{cases}
\]
In particular, if  $a_j (\ell, n , \lambda) > 0$ for all $j = 0,\ldots,i-1$ and $a_i(\ell,n,\lambda) \geq 0$, then
\[
a_i (\ell, n , \lambda) = \dim_{\k} [A]_i. 
\]    
\end{thm}

\begin{proof} 
To simplify notation set $a_j = a_j (\ell, n , \lambda)$. 

First consider the case where $2\ell \le n$. Then the proof of  Theorem  \ref{thm:dim formula} gives, for all $i \le d$, 
\[
\dim_{\k} [A]_i = a_i, 
\]
and so the conclusion holds without the hypothesis that $a_j>0$ for $j < i$.  Note that, in this case, $a_j \geq 0$ for all $j\leq i$.

Now assume that $2\ell > n$. We have seen in Remark \ref{formula} that
\begin{align*}
\HS(B) &  =  \left [ \HS(S/I_{P_1}) \right]^\ell  \\[1ex]
& = \frac{1}{(1-t)^{\ell n}} \left [  1-\sum^r_{i=1} t^{d-d_i} +(r-1)t^d       \right] ^\ell .
\end{align*}
Since  the elements of $\cL'$ provide a regular sequence in $B$ of length $\ell (n-2)$ by Proposition \ref{prop:improper inters}, we get 
\[
\HS (B/ \cL' B) = \frac{1}{(1-t)^{2 \ell}} \left [  1-\sum^r_{i=1} t^{d-d_i} +(r-1)t^d       \right] ^\ell .
\]
Hence Lemmas \ref{sop} and \ref{lem:char k-WLP} together with the WLP-Conjecture give 
\begin{equation}
     \label{eq:Hilb series predicted by WLP}
\HS (A) = \left | \frac{1}{(1-t)^{n}} \left [  1-\sum^r_{i=1} t^{d-d_i} +(r-1)t^d       \right] ^\ell \; \right |^+.
\end{equation}
Define integers $b_j$ by 
\[
\sum_{j \ge 0} b_j t^j = \frac{1}{(1-t)^{n}} \left [  1-\sum^r_{i=1} t^{d-d_i} +(r-1)t^d       \right] ^\ell. 
\]
Then computations as in the proof of Theorem  \ref{thm:dim formula} provide for $j \le d$, 
\begin{align*}
b_j  = \ & \binom{j+n-1}{n-1} - \ell \sum^r_{i=1} \binom{j + d_i - d+ n-1}{n-1}  + (r-1) \ell \binom{j}{d} \\
& \ +  \sum_{k=2}^\ell (-1)^k \binom{\ell}{k} \binom{j - k (d - d_1)+ n-1}{n-1} \\[1ex]
& \ +  \binom{\ell}{2} \binom{j + 2 d_2 -  2 d + n-1}{n-1} + \ell (\ell -1)  \binom{j + d_1 + d_2 - 2 d+ n-1}{n-1} \\
= & \ a_j. 
\end{align*}
Thus, we conclude for all non-negative integers $i \le d$
\[
\dim_{\k} [A]_i = \begin{cases}
0 & \text{if $a_j \le 0$ for some  $j$ with $0 \le j \le i$} \\
a_i > 0 & \text{otherwise}.  
\end{cases}
\]
\end{proof}

Notice that when $ 2\ell < n$ the intersection of the varieties determining tangent spaces to $\X_{n-1, \lambda}$ is non-empty, and the $a_j(\ell,n,\lambda)$ give its Hilbert function. When $ 2\ell = n$, the intersection of the  varieties determining tangent spaces to $\X_{n-1, \lambda}$ becomes empty, but this is still a proper intersection and so the methods of Section 3 continue to apply and result in the values given by the $a_j(\ell,n, \lambda)$. 
As soon as $2\ell > n$, however,  this intersection remains empty but becomes improper. Nevertheless, the $a_j(\ell,n, \lambda)\geq 0$ essentially provide the Hilbert function of the ``algebraic intersection," i.e. give the Hilbert function of $S/(I_{P_1}+\cdots +I_{P_\ell})$, as formalized in the previous result.

We now have the following extension of Theorem  \ref{thm:dim formula}: 

\begin{thm}
   \label{thm: gen dim formula} 
Let  $\lambda = [d_1,\ldots,d_r]\vdash d$  be a partition with $r \ge 2$. Assume that the WLP-Conjecture is true for some $\ell$ and $n$. Then:  
\begin{itemize}
\item[(a)] The secant variety $\sigma_{\ell} (\X_{n-1,\lambda})$ does not fill its ambient space if and only if 
\[
a_j (\ell, n , \lambda) > 0 \quad \text{ for all } j = 0,\ldots,d.
\]

\item[(b)] If $\sigma_{\ell} (\X_{n-1,\lambda})$ does not fill its ambient space, then it has dimension
\begin{align*}
\dim \sigma_{\ell}  (\X_{n-1, \lambda}) = & \ \ell \cdot \dim \X_{n-1, \lambda} + \ell -1    \\[1ex]
& \hspace*{0.3cm} {\displaystyle -  \sum_{k=2}^\ell (-1)^k \binom{\ell}{k}  \binom{d_1 - (k-1) (d_2 + \cdots + d_r) + n-1}{n-1} }  \\[1ex]
& \hspace*{0.3cm} {\displaystyle -  \binom{\ell}{2} \binom{2 d_2 - d + n-1}{n-1} - \ell (\ell -1)  \binom{d_1 + d_2 - d+ n-1}{n-1} }
\end{align*} 
\end{itemize}
\end{thm}

\begin{proof} 
Using Theorem \ref{thm:explain a-j} and 
\[
\dim \sigma_{\ell} (\X_{n-1,\lambda}) = \binom{d+n-1}{n-1} - 1 - \dim_{\k} [A]_d, 
\]
this follows from a computation as in the end of the  proof of Theorem \ref{thm:dim formula}. 
\end{proof}

\begin{rem}
    \label{rem:comp with {CGGS}}
(i) The argument in the proof of Theorem \ref{thm:explain a-j} shows more generally that the WLP-Conjecture allows us to determine the Hilbert function of the ring $A$ in every degree.    However, for finding the dimension of $\sigma_{\ell}  (\X_{n-1, \lambda})$, it is enough to know this Hilbert function in degree $d$ only. Thus, even if the WLP-conjecture is not true, it is possible that  Conjecture \ref{conj:main conj} is correct.  

(ii) As noted above, the WLP-Conjecture is true if $2\ell \leq n$.  Hence, Theorem \ref{thm: gen dim formula}  shows Conjecture \ref{conj:main conj} is true if $2\ell \le n$, thus proving Theorem \ref{thm: main thm - intro}(a). We will establish further instances of Conjecture \ref{conj:main conj} in Section \ref{sec:r is 2}. 

(iii) Complete results for  the dimension of $\sigma_{\ell}  (\X_{n-1, \lambda})$ have been previously  obtained only if $n=3$,   $\lambda = [1,\ldots,1]$ in \cite{A}, or $n=3$, $\ell = 2$ in \cite{CGGS}. Both results confirm Conjecture \ref{conj:main conj}. Moreover, if $\lambda = [1,\ldots,1]$ and $d \ge 3$, then $\sigma_{\ell}  (\X_{n-1, \lambda})$ is not defective, as predicted in   Conjecture  \ref{conj:defectivity conj}. The case $\lambda = [1,1]$ is covered by Theorem \ref{thm: main thm - intro}. The case $\ell = 2$ was discussed in Section \ref{sec:secant lines}. 

Thus, Conjecture \ref{conj:main conj} presents a unified formula for $\dim \sigma_{\ell}  (\X_{n-1, \lambda})$ in all cases. It is consistent with all the known results  that we have checked.
 
\end{rem}

We explore some consequences of our main conjecture, Conjecture \ref{conj:main conj}.

As we show in our next result, Conjecture \ref{conj:defectivity conj}(a) 
is an immediate consequence of Conjecture \ref{conj:main conj}
and thus holds in the many cases for which we establish Conjecture \ref{conj:main conj}.
The situation for Conjecture \ref{conj:defectivity conj}(b) is more complicated.
For certain choices of the parameters $n$, $\ell$ and $\lambda$, our next result shows
that Conjecture \ref{conj:defectivity conj}(b) is true while for some others
it shows that Conjecture \ref{conj:main conj} implies Conjecture \ref{conj:defectivity conj}(b).
In the remaining cases, for each $d_2 \geq \cdots\geq d_r>0$ and $\ell \geq n$,
it shows that there are at most finitely many cases,
namely $s =d_2 + \cdots + d_r \leq d_1 < (n-1)(s-1)$, for which
we do not know either that
Conjecture \ref{conj:defectivity conj}(b) is true or that 
Conjecture \ref{conj:main conj} implies Conjecture \ref{conj:defectivity conj}(b).
For these cases we have run numerical tests based on Proposition \ref{prop:non-defective}, as discussed in more detail below, 
which support our expectation that Conjecture \ref{conj:main conj} implies 
Conjecture \ref{conj:defectivity conj}(b) in these cases also.

\begin{prop}\label{prop:non-defective} 
As usual, let $n\geq 3$, $r\geq 2$, $N=\binom{n+d-1}{n-1}$ and 
$\lambda=[d_1,\ldots,d_r]$, where $d = d_1 + s$ and $s = d_2 + \dots + d_r$.
\begin{itemize}
\item[(a)] Assume $d_1 < s$ (and thus $r \ge 3$)). Then
Conjecture \ref{conj:main conj} implies Conjecture \ref{conj:defectivity conj}(a)
(i.e., that $\sigma_{\ell}  (\X_{n-1, \lambda})$ is not defective).
\item[(b)] Now assume $d_1 \geq s$.
\begin{itemize}
\item[(i)] If $2\ell \leq n$, then Conjecture \ref{conj:defectivity conj}(b) is true (i.e., either $\sigma_{\ell}  (\X_{n-1, \lambda})$
fills its ambient space, $\P^{N-1}$, or it is defective).
\item[(ii)] If $\frac{n}{2} < \ell \leq n$, then Conjecture \ref{conj:main conj} implies Conjecture \ref{conj:defectivity conj}(b).

\item[(iii)] If $d_1 < 2s$, then Conjecture \ref{conj:main conj} implies  Conjecture \ref{conj:defectivity conj}(b).

\item[(iv)] If $n\leq\ell$ and  $(n-1)(s-1)\leq d_1$, then $\sigma_{\ell}  (\X_{n-1, \lambda})$ fills its ambient space
(and hence Conjecture \ref{conj:defectivity conj}(b) is true).
\end{itemize}
\item[(c)] If $\ell\geq \binom{s+n-1}{n-1}$, then Conjecture \ref{conj:main conj} is true and $\sigma_{\ell}  (\X_{n-1, \lambda})$ fills its ambient space
(hence Conjecture \ref{conj:defectivity conj} is true).  

\item[(d)] 
If $\ell\geq \binom{s+n-1}{n-1}/t$, where $t$ is the number of occurrences of $d_1$ in $\lambda$, 
then Conjecture~\ref{conj:main conj} implies that $\sigma_{\ell}  (\X_{n-1, \lambda})$ fills its ambient space (and 
 hence, for such $\ell$, if Conjecture~\ref{conj:main conj} is true, then so is 
 Conjecture \ref{conj:defectivity conj}.)
\end{itemize}     
\end{prop}

\begin{proof}
(a) The assumptions  imply that the second and third lines of the formula in Conjecture \ref{conj:main conj}(b) are zero, so
\[
\dim \sigma_{\ell}  (\X_{n-1, \lambda}) =  \ell \cdot \dim \X_{n-1, \lambda} + \ell -1 =  
\min \left \{ \binom{d+n-1}{n-1} - 1, \;  \ell \cdot \dim \X_{n-1, \lambda} + \ell -1   \right \}, 
\]
which is  the expected dimension. 

(b)(i) This follows from Theorem \ref{thm:dim formula}.

(b)(ii) By definition, $\sigma_{\ell}  (\X_{n-1, \lambda})$ is not defective if it fills its ambient space, $\P^{N-1}$,
so assume $\sigma_{\ell}  (\X_{n-1, \lambda})$ does not fill its ambient space. Then 
$\dim \sigma_{\ell}  (\X_{n-1, \lambda})<N-1$ and Conjecture \ref{conj:main conj}(b)  gives 
\begin{equation}\label{Conj1.1b formula}
\begin{split}
\dim \sigma_{\ell}  (\X_{n-1, \lambda})  =\ &  \ell \cdot \dim \X_{n-1, \lambda} + \ell -1   \\
 &  - \sum_{k=2}^\ell (-1)^k \binom{\ell}{k}  \binom{d_1 - (k-1)s + n-1}{n-1} \\
 &  -  \binom{\ell}{2} \binom{2 d_2 - d + n-1}{n-1} - \ell (\ell -1)  \binom{d_1 + d_2 - d+ n-1}{n-1} .
\end{split}
\end{equation}
Let $\Syz$ be the first syzygy module of a complete intersection in $S$ that is generated by $\ell \le n$  forms of 
degree $d-d_1 = d_2 +\ldots + d_r$. We observed in Remark \ref{rem:syzygy interpretation} that 
\[
\sum_{k=2}^\ell (-1)^k \binom{\ell}{k}  \binom{d_1 - (k-1)s + n-1}{n-1} = \dim_{\k} [\Syz]_d. 
\]
Hence we get 
\[
\dim \sigma_{\ell}  (\X_{n-1, \lambda}) \le \ell \cdot \dim \X_{n-1, \lambda} + \ell -1 - \dim_{\k} [\Syz]_d.  
\]
The Koszul resolution shows that the initial degree of $\Syz$ is $2 (d-d_1)$. Since $\Syz$ is torsion free, it 
follows that $[\Syz]_d \neq 0$ if and only if $d \ge 2 (d - d_1)$, which is equivalent to 
\[
d_1 \ge d - d_1 = s. 
\]
Therefore our assumption gives 
\[
\dim \sigma_{\ell}  (\X_{n-1, \lambda}) < \ell \cdot \dim \X_{n-1, \lambda} + \ell -1,
\]
so $\dim \sigma_{\ell}  (\X_{n-1, \lambda})< \hbox{exp.dim } \sigma_{\ell} (\X_{n-1,\lambda})$, and we are done. 

(b)(iii)  As in the proof of (b)(ii), if $\sigma_{\ell}  (\X_{n-1, \lambda})$ does not fill its ambient space 
(and so $\dim \sigma_{\ell}  (\X_{n-1, \lambda})<N-1$) we must show that
it is defective. Put 
\[
g=\sum_{j=2}^\ell (-1)^j \binom{\ell}{j}  \binom{d_1 - (j-1)s + n-1}{n-1}.
\]
If $s \le d_1 < 2s$, then $g >0$. Thus, the summation in display \eqref{Conj1.1b formula} is positive,  but it is subtracted so we have
$\dim \sigma_{\ell}  (\X_{n-1, \lambda})  <  \ell \cdot \dim \X_{n-1, \lambda} + \ell -1$, and hence
$\sigma_{\ell}  (\X_{n-1, \lambda})$ is defective. Thus in the presence of the restriction on $d_1$, Conjecture \ref{conj:main conj} implies Conjecture \ref{conj:defectivity conj}(b).

(b)(iv) To see that $\sigma_{\ell}  (\X_{n-1, \lambda})$ fills its ambient space for $d_1\geq (n-1)(s-1)$ and
$\ell\geq n$, it is enough to do so for $\ell=n$. Take $\ell$ general points $P_j$ on $\X_{n-1, \lambda}$. Each ideal $I_{P_j}$ 
contains a minimal generator of degree $d - d_1 = s$. Thus, by genericity, the ideal $I=I_{P_1}+\cdots+I_{P_\ell}$ 
contains a complete intersection generated by $n$ forms of degree $s$. The socle degree of this complete intersection is $n(s-1)$. Thus $[R/I]_d=0$ if
$d_1+s=d>n(s-1)$; i.e., if $d_1\geq (n-1)(s-1)$.  It follows that for these $\ell$ and $d_1$ we get that Conjecture \ref{conj:defectivity conj}(b) is true.

(c) Note that $a_s = a_s(\ell, n, \lambda) =\binom{s+n-1}{n-1} -t\ell$, hence $a_s\leq0$ for  $\ell\geq \binom{s+n-1}{n-1}/t$, and  so also for $\ell \geq \binom{s+n-1}{n-1}$. So with the latter hypothesis, to prove Conjecture \ref{conj:main conj} we must show that $\sigma_{\ell}  (\X_{n-1, \lambda})$ fills its ambient space. Notice that $[S]_s$ has a basis consisting of monomials, hence a basis of forms each of which factors as a product of forms of degrees $d_2,\dots,d_r$.  Thus we can find points $P_1,\dots,P_\ell$ for which $I=I_{P_1}+\cdots+I_{P_\ell}$ spans $[S]_s$, so also in degree $d$ we have $[I]_d = [S]_d$. The same is then true for a general choice of $\ell$ points, and so $\sigma_{\ell}  (\X_{n-1, \lambda})$ indeed fills its ambient space. But then for these $\ell$, all parts of Conjecture  \ref{conj:defectivity conj} are automatically true as well.

(d) Finally, assume that $\ell\geq \binom{s+n-1}{n-1}/t$. Since $a_s \leq 0$, Conjecture \ref{conj:main conj} would imply that
$\sigma_{\ell}  (\X_{n-1, \lambda})$ fills its ambient space and so
Conjecture \ref{conj:defectivity conj} would hold.  
(We also note that the bound $\ell\geq \binom{s+n-1}{n-1}/t$ is sharp,
since $\sigma_{\ell}  (\X_{n-1, \lambda})$
does not always fill its ambient space for $\ell<\binom{s+n-1}{n-1}/t$, as we see for 
$\sigma_{\ell}  (\X_{n-1, [1,1]})$ by Theorem \ref{thm:defectivity - intro}(c).)
\end{proof}

\begin{rem}\label{NumericalTestingRemark}
 Proposition \ref{prop:non-defective}
is the basis for numerical tests that support our belief that all of the cases left open
do in fact follow from Conjecture \ref{conj:main conj}. In these cases, we have $n < \ell$ and $2s\leq d_1<(n-1)(s-1)$. As noted above, Conjecture \ref{conj:main conj} predicts: If $\sigma_{\ell}  (\X_{n-1, \lambda})$ does not  fill its ambient space, then 
\begin{equation*}\begin{split}
\dim \sigma_{\ell}  (\X_{n-1, \lambda})  =\ &  \ell \cdot \dim \X_{n-1, \lambda} + \ell -1  - g  \\
 &  -  \binom{\ell}{2} \binom{2 d_2 - d + n-1}{n-1} - \ell (\ell -1)  \binom{d_1 + d_2 - d+ n-1}{n-1} .
\end{split}
\end{equation*}

Thus, Conjecture \ref{conj:defectivity conj}(b) follows if $g > 0$. 

We have used Macaulay2 \cite{Macaulay2}
to check for all cases satisfying the above restrictions with 
$n,\ell,s\leq 60$ that Conjecture \ref{conj:main conj} implies  Conjecture \ref{conj:defectivity conj}(b).
There were 57,345,933 such cases. Note that typically for each $s$ and $d_1$, there are many possible 
partitions of $d=d_1+s$; thus for each of the 57,345,933 cases for which $g\leq0$, we merely checked that
$\ell\dim(\X_{n-1,[d_1,1,\ldots,1]})+\ell\geq N$, and hence by Proposition \ref{prop:non-defective}(b)(iii)
we see that Conjecture \ref{conj:main conj} implies Conjecture \ref{conj:defectivity conj}(b) 
for each case when $\lambda=[d_1,1,\ldots,1]$. But by 
Corollary \ref{cor:dim comparison}
this means
Conjecture \ref{conj:main conj} implies Conjecture \ref{conj:defectivity conj}(b) 
also for all other partitions $\lambda$ of $d=d_1+s$ for each of these 57,345,933 cases.
\end{rem}

\begin{cor}\label{Thm1.2c}
Let $3\leq n\leq \ell\leq 1+\frac{d_1+n-1}{s}$ with $\lambda=[d_1,\ldots,d_r] \vdash d$, $r \geq 3$, and $s=d_2+\cdots+d_r$.
Then Conjecture \ref{conj:main conj} is true for such $n$, $\ell$ and $\lambda$.
\end{cor}

\begin{proof} 
Conjecture \ref{conj:main conj}(a) asserts that $\sigma_\ell(\XX_{n-1,\lambda})$ fills its ambient space 
if and only if $a_j(\ell,n,\lambda)$ 
 is not positive for some  integer $j$ with  $s\leq j\leq d$,
while Conjecture \ref{conj:main conj}(b) applies only when $\sigma_\ell(\XX_{n-1,\lambda})$ does not fill its ambient space. 
Since $d_1 \geq (n - 1)(s-1)$ is equivalent to $1 + \frac{d_1+n-1}{s} \geq n$,  Proposition \ref{prop:non-defective} implies $\sigma_\ell(\XX_{n-1,\lambda})$ fills its ambient space. Thus
Conjecture \ref{conj:main conj} is true if we show $a_d(\ell,n,\lambda)\leq 0$.

Recall the identity 
\[
\sum_{k=2}^\ell (-1)^k \binom{\ell}{k}  \binom{d - k s + n-1}{n-1}=0.
\]
(See formula 10.13 of \url{http://www.math.wvu.edu/~gould/Vol.4.PDF},
where $n,k,r,y,x$ in 10.13 become, respectively, our
$\ell$, $j$, $n-1$, $d+n-1$ and $-s$, so the assumption $n>r$ in 10.13 
becomes $\ell > n-1$ and is thus satisfied. We also note that 10.13 does not assume
that $\binom{a}{b}=0$ when $a<0$, but our assumption $\ell\leq 1+\frac{d_1+n-1}{s}$ 
is equivalent to $d-\ell s + n-1\geq 0$. This ensures that $d-js+n-1\geq0$ hence
the convention used in 10.13 agrees with our convention that    
$\binom{d-js+n-1}{n-1}=0$ when $d-js+n-1<n-1$.)

Using the identity above, we have
$$\sum_{j=2}^\ell(-1)^j\binom{\ell}{j}\binom{d-js+n-1}{n-1}=-\binom{d+n-1}{n-1}+\ell\binom{d_1+n-1}{n-1},$$
and substituting this into the expression for $a_d(\ell,n,\lambda)$
given in Definition \ref{def:integers a-j} we obtain
$$a_d(\ell,n,\lambda)=-\ell\sum_{i>1}\Bigg(\binom{d_i+n-1}{n-1}-1\Bigg)<0.$$
\end{proof}

We are now ready to prove one of the main results of the paper, as mentioned in the introduction.

\begin{proof}[Proof of Theorem \ref{thm:defectivity - intro}]
Part (a) follows from Theorem \ref{general r geq 3}(a). 
Part (b) follows from Theorem \ref{general r=2} and Theorem \ref{general r geq 3}(b).
Part (c) follows from Proposition \ref{prop:non-defective}(iv) and Theorem \ref{thm:dim formula}.
\end{proof}

We conclude this section with a case where we can show that the prediction of Conjecture~\ref{conj:main conj} is at least an  upper bound. 

\begin{prop}
     \label{prop:upper bound}
If $2 \ell = n+1$, then the dimension of $\sigma_{\ell}  (\X_{n-1, \lambda})$ is at most the number predicted in Conjecture \ref{conj:main conj}.     
\end{prop}

\begin{proof}
Using the above notation, put $A' = B/\cL'$. Let $L \in A'$ be a general linear element. Notice that the assumption on $\ell$ gives $|\cL | = |\cL' | + 1$. Thus, we get 
\begin{equation}
         \label{eq:upper bound}
\HS (A) = \HS (B/\cL B) = \HS (A'/L A') \ge \left | (1-t) \HS (A') \right |^+,
\end{equation}
which means that the comparison is true coefficientwise.

If $a_j (\ell, n, \lambda) \le 0$ for some non-negative $j \le d$, then Conjecture \ref{conj:main conj}  says that $\sigma_{\ell} (\X_{n-1, \lambda})$ fills its ambient space, and so the estimate follows. 

Otherwise, Equation \eqref{eq:upper bound} implies 
\[
\dim_{\k} [A]_d \ge a_d (\ell, n, \lambda), 
\]
which gives 
\[
\dim \sigma_{\ell}  (\X_{n-1, \lambda}) = \binom{d + n-1}{n-1} - 1 - \dim_{\k} [A]_d \le \binom{d + n-1}{n-1} - 1 - a_d (\ell, n, \lambda).
\]
Since, the right-hand side is the formula for $\dim \sigma_{\ell}  (\X_{n-1, \lambda})$ that is predicted by Conjecture \ref{conj:main conj}, this completes the argument. 
\end{proof}

\begin{rem}
    \label{rem:upper bound} 
Consider an arbitrary graded $\k$-algebra, and let $L_1, L_2 \in [A]_1$ be two general elements. Then it is \emph{not} necessarily true that 
\[
\HS (A/(L_1, L_2) \ge \left | (1-t)^2 \cdot \HS (A') \right |^+.
\]   
Thus, the above argument cannot be easily extended to show that Conjecture \ref{conj:main conj} gives an upper bound for  $\dim \sigma_{\ell}  (\X_{n-1, \lambda})$ for all $\ell \ge \frac{n+1}{2}$. Note however that in the following section we will prove that Conjecture \ref{conj:main conj} does give an upper bound if $r = 2$ by using a different approach. 
\end{rem}


\section{Forms with two factors  and Fr\"oberg's Conjecture} 
   \label{sec:r is 2}

In this section we focus on the case $r = 2$, that is, we consider secant varieties to the varieties whose general point corresponds to a product of two irreducible polynomials. 
We begin  by recalling that the dimension of secant varieties, in case $r = 2$, is related to a famous conjecture of Fr\"oberg (see \cite{CCG:1}). We systematically relate this conjecture to our approach in the previous section. In particular, we will see that Fr\"oberg's Conjecture and the WLP-Conjecture lead to the same prediction for the dimension of the secant variety in the case $r = 2$.  This allows us to establish further instances of Conjecture~\ref{conj:main conj}. 

Fr\"oberg's Conjecture concerns the Hilbert function of an ideal generated by generic forms. More precisely, it says: 

\begin{conj}[Fr\"oberg's Conjecture \cite{froeberg}] 
     \label{conj:Froeberg} 
Let $J \subset S = \k[x_1,\ldots,x_n]$ be an ideal generated by $s$ generic forms of degrees $e_1,\ldots,e_s$ in $S$. Then the Hilbert series of $S/J$ is 
\[
\HS (S/J) = \left | \frac{\prod_{i=1}^s (1 - t^{e_i})}{(1-t)^n}    \right |^+ . 
\]     
\end{conj}

There is an equivalent version of Fr\"oberg's Conjecture that gives a recursion to predict the Hilbert function of such an algebra. 

\begin{conj}[Fr\"oberg's Conjecture, recursive version] 
     \label{conj:Froeberg recursive} 
Let $J \subset S = \k[x_1,\ldots,x_n]$ be an ideal that is generated by generic forms, and let $f \in S$ be a generic form of degree $e$. Then, for all integers $j$, 
\[
\dim_{\k} [S/(J, f)]_j = \max \{0, \;  \dim_{\k}   [S/J]_j -   \dim_{\k}  [S/J]_{j - e} \}. 
\]     
\end{conj}     

Comparing the latter version with Definition \ref{def:WLP} shows that $S/J$ has the \WLP\ if the above is true and $e=1$. 
We refer to \cite[Proposition 2.1]{MMN3} for further results on the relation between Fr\"oberg's 
conjecture and the Lefschetz properties. Here we need only the following observation. 

\begin{prop}
     \label{prop:Froeberg gives WLP}
Assume Fr\"oberg's Conjecture is true for polynomial rings in up to $n$ variables. If $J \subset S = \k[x_1,\ldots,x_n]$ is 
an ideal that is generated by at least $n$ generic forms, then $S/J$ has the $n$-WLP.      
\end{prop}

\begin{proof}
Let $L \in S$ be a generic linear form. Then,  as noted above, $L$ is a Lefschetz element for $S/J$. 
Since $S/(J, L)$ is isomorphic to a quotient of a polynomial ring in $n-1$ variables modulo an ideal 
generated by generic forms in these $(n-1)$ variables, we can use this argument  $n$ times. 
\end{proof} 

Since each quotient of a polynomial ring in at most two variables has the \WLP\ by \cite{HMNW}, 
the properties $n$-WLP and $(n-2)$-WLP are equivalent if $n \ge 2$. 

We now relate this to the secant  varieties of $\X_{n-1, \lambda}$, where $\lambda$ is a partition with two parts. 
In this case, we simplify notation and write $\lambda = [d- k,k]$, where $1 \le k \le \frac{d}{2}$.

Our starting point for the case $r = 2$ is the following observation. 

\begin{lem}
   \label{lem:dim for r=2}
If $\lambda = [d- k,k]$, then    
\[
\dim \sigma_{\ell}  (\X_{n-1, \lambda})  = -1 + \dim_{\k} [I]_d,
\]
where $I \subset S$ is an ideal generated by $\ell$ generic forms of degree $d-k$ and $\ell$ generic forms of degree $k$.
\end{lem} 

\begin{proof}
This follows from Corollary \ref{dim sigma ell} and Proposition \ref{nicer}. 
\end{proof} 

In the case $r=2$, the definition of the integers $a_j (\ell, n, \lambda)$ becomes somewhat simpler. 

\begin{rem}
    \label{rem:simples a-j for r is 2} 
Assume $\lambda = [d- k,k]$, where $1 \le k \le \frac{d}{2}$. Then 
\begin{align*}
a_j (\ell, n, \lambda)  = &  \binom{j+n-1}{n-1} - \ell \left [ \binom{j + k - d+ n-1}{n-1}  + \binom{j - k + n-1}{n-1} \right ]  \\
&  +  \sum_{i=2}^\ell (-1)^i \binom{\ell}{i} \binom{j - i k + n-1}{n-1} 
\end{align*}     
if $0 \le j < d$, and 
\begin{align*}
a_d  (\ell, n, \lambda) = &  \binom{d+n-1}{n-1} - \ell \left [ \binom{k + n-1}{n-1}  + \binom{d - k + n-1}{n-1} \right ]  \\
&  +  \sum_{i=2}^\ell (-1)^i \binom{\ell}{i} \binom{d - i k + n-1}{n-1} + \binom{\ell}{2} \binom{2 k - d +n -1}{n-1} + \ell^2. 
\end{align*}  
Observe that the penultimate term is zero, unless $k = \frac{d}{2}$. 
\end{rem}

Lemma \ref{lem:dim for r=2} allows us to relate Fr\"oberg's Conjecture to our work in the previous sections. 

\begin{prop}
      \label{prop:Froeberg gives main conj} 
Let $\lambda = [d-k, k]$ be a partition with two parts. Then, for each $\ell \ge 2$ and each $n \ge 3$, the value for the dimension of the secant variety in Conjecture \ref{conj:main conj} gives an upper bound for the dimension of $\sigma_{\ell} (\X_{n-1, \lambda})$. 
Moreover, if Fr\"oberg's Conjecture is true for 
$S$, then, for all $\ell \ge 2$, the variety  $\sigma_{\ell} (\X_{n-1, \lambda})$ has the dimension predicted in Conjecture \ref{conj:main conj}.
\end{prop}

\begin{proof} 
If $\ell \le \frac{n}{2}$, Conjecture \ref{conj:main conj} is  true  by Theorem \ref{thm:dim formula}. Thus, we may assume $2 \ell > n$. 

With $I$ as in Lemma \ref{lem:dim for r=2}, we have 
\[
\dim \sigma_{\ell} (\X_{n-1,\lambda}) = \binom{d+n-1}{n-1} - 1 - \dim_{\k} [S/I]_d.
\]
The value predicted for $\dim \sigma_{\ell} (\X_{n-1,\lambda})$ by Conjecture \ref{conj:main conj}
comes, by Theorem \ref{thm:explain a-j}, from the value of $\dim_{\k} [S/I]_d$ predicted by the 
WLP-Conjecture \ref{conj:WLP}. Thus it is enough to derive from Fr\"oberg's Conjecture \ref{conj:Froeberg}
the same value for $\dim_{\k} [S/I]_d$ as given by Conjecture \ref{conj:WLP}, and to show that this 
value is a lower bound for the actual value.

Using our earlier notation, observe that $B/\cL' B \cong U/J$, where $U$ is a polynomial ring in $2 \ell$ 
variables and $J$ is a complete intersection generated by $\ell$ general forms of degree $k$ and $\ell$ 
general forms of degree $d-k$. Hence the Hilbert series of $U/J$ is 
\[
\HS (U/J) = \frac{(1 - t^{d-k})^\ell (1-t^k)^\ell}{(1-t)^{2 \ell}} . 
\]     
Lemma \ref{sop}  shows that $S/I$ is isomorphic to $B/\cL B$, which is obtained from $U/J$ by quotienting by $(2 \ell - n)$ 
general linear forms. Hence \cite[Theorem on  p.\ 120, \S1]{froeberg} gives
\[
\HS (S/I) \geq  \left | \frac{(1 - t^{d-k})^\ell (1-t^k)^\ell(1-t)^{2\ell-n}}{(1-t)^{2\ell}}    \right |^+=\left | \frac{(1 - t^{d-k})^\ell (1-t^k)^\ell}{(1-t)^n}    \right |^+,
\]    
the right hand side of which is exactly the Hilbert series of $S/I$ as predicted by Fr\"oberg's Conjecture \ref{conj:Froeberg}.
Moreover, the WLP-Conjecture \ref{conj:WLP} predicts this same Hilbert series for $S/I$ (see Equation \eqref{eq:Hilb series predicted by WLP}),
and hence equality holds if Fr\"oberg's Conjecture does.
\end{proof} 

\begin{rem}
Notice that in Proposition \ref{prop:Froeberg gives main conj} we assumed the correctness of Fr\"oberg's Conjecture only for ideals in $S$. If we assume more, namely that this conjecture is true for all ideals in $\max \{n, \ 2 \ell \}$ variables, then Proposition \ref{prop:Froeberg gives WLP} shows that $B/\cL' B$  has the $(2 \ell-n)$-\WLP. Hence, in this case  Theorem \ref{thm: gen dim formula} immediately gives the conclusion of the above proposition. 
\end{rem} 

We are ready to establish new instances where Conjecture \ref{conj:main conj} holds.

\begin{thm}
    \label{thm:dim when froeberg}
Conjecture \ref{conj:main conj} is true if $r=2$ and     
\begin{itemize}
\item[(a)] $2\ell \le n+1$ \quad or 
\item[(b)] $n = 3$ \quad or 
\item[(c)] $\lambda = [1, 1]$, that is, $d = 2$.  
\end{itemize}
\end{thm}

\begin{proof}
We use Proposition \ref{prop:Froeberg gives main conj}. 
Fr\"oberg's Conjecture is true for forms in at most three variables by a result of Anick \cite{anick}. This gives (b). The conjecture also holds for ideals generated by general linear forms, and thus (c) follows. 

Turning to  (a), by Theorem \ref{thm:dim formula} it suffices to consider the case where $2 \ell = n + 1$. Then $I$ is generated by $n+1$ general forms in $n$ variables. For such ideals Fr\"oberg's Conjecture is true because complete intersections of general forms have the \SLP\ (see \cite{stanley, watanabe, RRR, sekiguchi} or \cite{HP}).  
\end{proof}

\begin{rem}
With some work, part (a) of Theorem \ref{thm:dim when froeberg} can be shown to be equivalent to Theorem 5.1 of \cite{CCG:1}. There, however, all the cases where $\sigma_\ell(\X_{n-1,\lambda})$ fills its ambient space are enumerated.
\end{rem}

We now begin working out more explicit formulas for some particular partitions as consequences of Proposition \ref{prop:Froeberg gives main conj}. First we consider balanced partitions. 

\begin{thm} 
      \label{thm:balanced}
Consider a balanced partition $\lambda  =  [\frac{d}{2},\frac{d}{2}]$ and fix $n \geq 3$. Assume that Fr\"oberg's Conjecture holds for the polynomial ring $S = \k[x_1,\ldots,x_n]$. 
Put
\[
\ell_0 = \frac{1}{2} \left [ \binom{\frac{d}{2} + n-1}{n-1} + \frac{1}{2} - \sqrt{\left [ \binom{\frac{d}{2} + n-1}{n-1} + \frac{1}{2} \right ]^2 - 2 \binom{d+n-1}{n-1}} \ \right ].
\]
Then
\begin{multline*}
\dim \sigma_\ell(\X_{n-1, [\frac{d}{2},\frac{d}{2}]})
 =  \begin{cases}
\ell \cdot \dim \X_{n-1,\lambda} + \ell-1 - 2 \ell (\ell - 1)  <  \binom{d+n-1}{n-1} - 1 & \text{ if } 2 \le \ell < \ell_0 \\[2ex]
\binom{d+n-1}{n-1} - 1 & \text{ if } \ell_0 \le \ell.
\end{cases}
\end{multline*}
In particular, the secant variety $\sigma_\ell(\X_{n-1,\lambda})$ is defective if and only if it does not fill its ambient space.
Furthermore, the defect is
\begin{align*}
\delta_\ell & =  \begin{cases}
2 \ell (\ell - 1) & \text{ if } 2 \le \ell \le {\displaystyle \frac{\binom{d + n-1}{n-1}}{2 \binom{\frac{d}{2} + n-1}{n-1} - 1}} \\[4ex]
\binom{d+n-1}{n-1} - 1 - \dim \sigma_\ell(\X_{n-1, [\frac{d}{2},\frac{d}{2}]})   & \text{ if } {\displaystyle \frac{\binom{d + n-1}{n-1}}{2 \binom{\frac{d}{2} + n-1}{n-1} - 1}} < \ell < \ell_0.
\end{cases}
\end{align*}
\end{thm}

\begin{proof}
Put $N = \binom{d+n-1}{n-1}$ and recall that
\[
\dim \X_{n-1,\lambda} = 2 \cdot \binom{\frac{d}{2}+n-1}{n-1}   -2 .
\]
Thus, the expected dimension of $\sigma_\ell (\X_{n-1,\lambda})$ is
\begin{align*}
\hbox{exp.dim } \sigma_\ell (\X_{n-1,\lambda}) &  =   \min \{ N-1,  \ell \cdot \dim \X_{n-1,\lambda} + (\ell -1) \} \\[1ex]
& = \min \left \{ N-1\ , \  2 \ell \cdot \binom{\frac{d}{2}+n-1}{n-1}  - \ell -1 \right \}.
\end{align*}
In particular,
\begin{equation}
   \label{eq:est exp dim}
\hbox{exp.dim } \sigma_\ell (\X_{n-1,\lambda}) = \ell \cdot \dim \X_{n-1,\lambda} + (\ell -1) \text{ if and only if } 2 \le \ell \le \frac{N}{2 \binom{\frac{d}{2} + n-1}{n-1} - 1}.
\end{equation}

We now will consider various ranges for the value of $\ell$ and use Lemma \ref{lem:dim for r=2}. For the partition $\lambda$, the ideal $I$ is generated by $2 \ell$ general forms of degree $\frac{d}{2}$. Instead of applying Proposition \ref{prop:Froeberg gives main conj} directly, it is more convenient to use the recursive approach (see Conjecture \ref{conj:Froeberg recursive}). 

Assume first $2 \le \ell < \frac{n}{2}$. Then Theorem  \ref{thm:dim formula} gives that $\sigma_\ell(\X_{n-1,\lambda})$ does not fill its ambient space and has dimension 
\begin{align*}
\dim \sigma_\ell(\X_{n-1,\lambda}) & = \ell \cdot \dim \X_{n-1,\lambda} + (\ell-1) - 2 \binom{\ell}{2} - \ell (\ell -1) \\[1ex]
& = \ell \cdot \dim \X_{n-1,\lambda} + (\ell-1) - 2 \ell (\ell -1).
\end{align*}
This proves the statement if $\ell < \frac{n}{2}$.

Assume now $\frac{n}{2} \le \ell$. In order to simplify notation, set $k = \frac{d}{2}$ and $t = 2 \ell - n \ge 0$. 

In this range of $\ell$, $S/I$ is artinian and
\[
\dim_{\k} [S/I]_k = \max \left \{0, \; \binom{k+n-1}{n-1} - 2 \ell \right \}.
\]
Hence $[S/I]_k = 0$ if $2 \ell \ge \binom{k+n-1}{n-1}$, which implies $[S/I]_d = 0$. It follows that $\sigma_\ell(\X_{n-1,\lambda})$ fills $\PP^{N-1}$ for such $\ell$.

We are left to consider $\ell$ such that $\frac{n}{2} \le \ell < \frac{1}{2} \binom{k+n-1}{n-1}$. Notice that this forces $k \ge 2$, that is, $d \ge 4$.

For $i = 0,1,\ldots,t = 2 \ell -n$, let $\mathfrak a_i$ be the ideal generated by $n+i$ general forms of degree $k$.
Observe that $\mathfrak a_0$ is a complete intersection and $I = \mathfrak a_{t}$.
Notice that, for all $i$,
\[
\dim_{\k} [S/\mathfrak a_i]_k  = \binom{k+n-1}{n-1} - n-i = \dim_{\k} [S/\mathfrak a_0]_k - i.
\]
The minimal free resolution of $S/\mathfrak a_0$ has the form
\[
\dots
\rightarrow
S(-d)^{\binom{n}{2}} \\
\rightarrow
S(-k)^n \\
\rightarrow S \rightarrow S/\mathfrak a_0 \rightarrow 0,
\]
where we only display the terms that are non-trivial in degree $d$. This shows
\[
\dim_{\k}  [S/\mathfrak a_0]_d = N -  n \cdot \binom{k+n-1}{n-1}  + \binom{n}{2}.
\]

Fr\"oberg's Conjecture \ref{conj:Froeberg recursive} predicts, for all $i$,
\[
\dim_{\k} [S/\mathfrak a_{i+1}]_d = \max \{ \dim_{\k} [S/\mathfrak a_i]_d - \dim_{\k} [S/\mathfrak a_i]_k, \ 0 \}.
\]
Hence, we get
\begin{align*}
\dim_{\k} [S/I]_d & =  \max \left  \{  0, \ \dim_\k [S/\mathfrak a_0]_d - t \cdot \dim_\k [S/\mathfrak a_0]_k + \binom{t}{2} \right \} \\[1ex]
& = \max \left \{ 0\ , \  \left [ N -  n \cdot \binom{k+n-1}{n-1}  + \binom{n}{2} \right ] - t \left [ \binom{k+n-1}{n-1} - n \right ] + \binom{t}{2} \right \} \\[1ex]
& =  \max \left \{ 0 \ , \  N - 2 \ell \cdot \binom{k+n-1}{n-1} + \binom{n}{2} + t n + \binom{t}{2}      \right \} \\[1ex]
& = \max \left \{ 0 \ , \  N - 2 \ell \cdot \binom{k+n-1}{n-1} + 2 \ell^2 - \ell \right \}.
\end{align*}
It follows that
\begin{align*}
\dim \sigma_\ell(\X_{n-1,\lambda}) & = \min \left  \{ N  -1 \ , \   2 \ell \binom{k+n-1}{n-1} + \ell-1 - 2 \ell^2 \right \}
\end{align*}
(note that we subtracted 1 from the dimension of the component of the ideal). 
Therefore, for $\ell < \frac{1}{2} \binom{k+n-1}{n-1}$, the variety $\sigma_\ell(\X_{n-1,\lambda})$ fills $\PP^{N-1}$ if and only if
\[
N \le  2 \ell \binom{k+n-1}{n-1} + \ell - 2 \ell^2,
\]
which means
\[
\ell \ge \frac{1}{2} \left [ \binom{k+ n-1}{n-1} + \frac{1}{2} - \frac{1}{2} \sqrt{\left [ \binom{k + n-1}{n-1} + \frac{1}{2} \right ]^2 - 2 N} \ \right ] = \ell_0.
\]
An induction on $n \ge 2$ shows that  the radicand is at least $\frac{1}{4}$, which also implies 
\[
\left \lceil \frac{1}{2} \binom{k+n-1}{n-1} \right \rceil \ge \ell_0. 
\]

We conclude that $\sigma_\ell(\X_{n-1,\lambda})$ fills its ambient space if and only if $\ell \ge \ell_0$ and that 
\begin{align*}
\dim \sigma_\ell(\X_{n-1,\lambda}) & =  2 \ell \binom{k+n-1}{n-1} + \ell-1 - 2 \ell^2 \\
& = \ell \cdot \dim \X_{n-1,\lambda} + \ell-1 - 2 \ell (\ell - 1)
\end{align*}
if $\frac{n}{2} \le \ell \le \ell_0$. This concludes finding the dimension of $\sigma_\ell(\X_{n-1,\lambda})$. Combining the result with Observation \eqref{eq:est exp dim} proves the assertion on the defect.
\end{proof}

Second, we consider the most unbalanced partition of $d$ into two parts. Notice that the following result  is true unconditionally. Since the partition $[1, 1]$ has been dealt with in the previous result (see Theorem \ref{thm:dim when froeberg}), there is no harm in assuming $d \ge 3$ in the next statement. 
The fact that $\sigma_\ell (\X_{n-1,[d-1, 1]})$ fills its ambient space if and only if $\ell \geq \ell_0$ was  shown in \cite[Proposition 5.6]{CCG:1}. The dimension of $\sigma_\ell (\X_{n-1,[d-1, 1]})$ can also be found in \cite[Proposition 4.4]{BCGI}. We give a new proof of these facts using our methods.

Note that, in the following theorem, the formula for $\dim(\X_{n-1,[d-1,1]})$ is simply the specialization of the formula of Conjecture \ref{conj:main conj} to the case at hand.

\begin{thm}
    \label{thm:case d-1, 1}
Assume $\lambda = [d-1, 1]$, where $d \ge 3$.  Put
\[
\ell_0 = 
\min \left \{ \ell \ge \frac{n}{2} \; | \; \ell \in \ZZ \text{ and } \binom{d - \ell + n - 1}{d} \le \ell (n- \ell) \right \}.
\]
Then $\ell_0 \le n-1$ and
\[
\dim \sigma_{\ell}  (\X_{n-1, [d-1, 1]})  =
\begin{cases}
\binom{d+n-1}{n-1} -   \binom{d+n- \ell - 1}{d}  + \ell (n- \ell) - 1 <  \binom{d+n-1}{n-1}  - 1 & \text{ if } 2 \le \ell < \ell_0 \\[.5ex]
\binom{d+n-1}{n-1}  - 1 & \text{ if }  \ell_0 \le \ell.
\end{cases}
\]
Moreover, the secant variety $\sigma_\ell(\X_{n-1,\lambda})$ is defective if and only if it does not fill its ambient space. In this case, the defect is 
\begin{align*}
\delta_\ell & =  \begin{cases}
\binom{d+n- \ell - 1}{d}  +    \ell \cdot  \binom{d+n-2}{n-1} - \binom{d+n-1}{n-1}+ \ell^2 \ge \ell^2 & \text{ if } 2 \le \ell \le {\displaystyle \frac{\binom{d + n-1}{n-1}}{\binom{d + n-2}{n-1} + n}} \\[4ex]
\binom{d+n-1}{n-1} - 1 - \dim \sigma_\ell(\X_{n-1, [\frac{d}{2},\frac{d}{2}]})   & \text{ if } {\displaystyle \frac{\binom{d + n-1}{n-1}}{\binom{d + n-2}{n-1} + n}} < \ell < \ell_0.
\end{cases}
\end{align*}
\end{thm}

\begin{proof} 
Again we use Lemma \ref{lem:dim for r=2}. This time
the ideal $I = I_{(1)} + \cdots +  I_{(\ell)}$ contains $\ell$ generic linear forms. Thus, $I = (x_1,\ldots,x_n)$ if $\ell \ge n$, and we are done in this case. If $\ell < n$, then we get
\[
A = S/I \cong \k[x_1,\ldots,x_{n-\ell}]/(G_1,\ldots,G_{\ell}),
\]
where each $G_j$ is a generic form of degree $d-1$ in $T = \k[x_1,\ldots,x_{n-\ell}]$. It follows that
\begin{align*}
\hspace{2em}&\hspace{-2em} \dim \sigma_{\ell} (\X_{n-1, \lambda})\\
& = -1 + \dim_{\k} [I]_d \\
& =  -1 + \dim_{\k} [S]_d - \dim_{\k} [A]_d \\
& =  -1 + \dim_{\k} [S]_d - \dim_{\k} [T]_d + \dim_{\k} [(G_1,\ldots,G_{\ell})]_d \\
& =  -1 + \binom{d+n-1}{n-1} - \binom{d+n- \ell - 1}{d} + \min \left \{\binom{d+n- \ell - 1}{d}, \ell (n- \ell) \right \} \\
& =  -1 + \binom{d+n-1}{n-1} -   \max \left \{0, \binom{d+n- \ell - 1}{d} - \ell (n- \ell) \right \}.
\end{align*}
In order to see the penultimate equality consider the graded minimal free resolution of $(G_1,\ldots,G_{\ell})$. Its beginning is of the form
\[
\cdots \to F \to T^{\ell} (-d+1) \to (G_1,\ldots,G_{\ell}) \to 0,
\]
where $F$ is a graded free $T$-module. It follows that  to compute $\dim_{\k} [(G_1,\ldots,G_{\ell})]_d$ it is enough to know the number of linearly independent linear syzygies of the ideal $(G_1,\ldots,G_{\ell})$. Since the forms $G_j$ are generic this number is the least possible by the main result in \cite{HL}, and the dimension formula follows. It shows that $\sigma_{\ell} (\X_{n-1, \lambda})$ fills its ambient space if and only if
\[
\binom{d+n - \ell - 1}{d} \le \ell (n - \ell).
\]
If $\ell = n-1$, this is true. Hence, the number $\ell_0$ is well defined and satisfies $\ell_0 \le n - 1$. Furthermore, if $\sigma_\ell(\X_{n-1,\lambda})$ 
fills its ambient space, then so does $\sigma_{\ell + 1} (\X_{n-1,\lambda})$.
This completes the argument for finding the dimension of $\sigma_{\ell} (\X_{n-1, \lambda})$.

It remains to discuss the defect of $\sigma_{\ell} (\X_{n-1, \lambda})$.
Note that
\[
\dim \X_{n-1, \lambda} = \binom{d+n-2}{n-1} + n -1,
\]
and so the expected dimension of $\sigma_{\ell} (\X_{n-1, \lambda})$ is
\[
\hbox{exp.dim } \sigma_\ell (\X_{n-1,\lambda}) = \min \left \{\binom{d+n-1}{n-1} - 1, \; \ell \cdot  \binom{d+n-2}{d-1} + \ell n - 1. \right \}.
\]
We need to show that $\sigma_{\ell} (\X_{n-1, \lambda})$ is defective if and only if $2 \le \ell < \ell_0 < n$. Since for such $\ell$ the variety $\sigma_{\ell} (\X_{n-1, \lambda})$ does not fill its ambient space, this is equivalent to proving
\[
\ell \cdot  \binom{d+n-2}{d-1} + \ell n - 1 > \binom{d+n-1}{n-1} -   \binom{d+n- \ell - 1}{d}  + \ell (n- \ell) - 1,
\]
that is,
\begin{equation}
    \label{eq:defect lin factor}
\binom{d+n- \ell - 1}{d} - \binom{d+n-1}{n-1} +    \ell \cdot  \binom{d+n-2}{d-1} > - \ell^2.
\end{equation}
Notice that
\begin{align*}
\binom{d+n- \ell - 1}{d} - \binom{d+n-1}{n-1} +    \ell \cdot  \binom{d+n-2}{d-1} &
= \sum_{j=2}^{\ell} (-1)^j \binom{\ell}{j} \binom{d-j+n-1}{n-1} \\
& = \dim_{\k} [\Syz]_d,
\end{align*}
where $\Syz$ is the first syzygy module of a complete intersection in $S$ that is generated by $\ell < n$ linear forms (see Remark \ref{rem:syzygy interpretation}). This shows that the left-hand side in Inequality \eqref{eq:defect lin factor} is non-negative, and hence establishes that this inequality is true.

In order to determine the positive defect, it is enough to observe that
\[
\hbox{exp.dim } \sigma_\ell (\X_{n-1,\lambda}) \le  \binom{d+n-1}{n-1} - 1
\]
if and only if
\[
\ell \le {\displaystyle \frac{\binom{d + n-1}{n-1}}{\binom{d + n-2}{n-1} + n}}.
\]
This concludes the calculation of the defect.
\end{proof}

Since we  discussed the case $\ell = 2$ in Section \ref{sec:secant lines}, we  illustrate the last result in the case $\ell = 3$. 

\begin{cor}
    \label{cor:l=3 r=2}
    Consider the secant plane variety $\sigma_{3} (\X_{n-1, [d-1, 1]})$.
\begin{itemize}
\item[(a)] $\sigma_{3} (\X_{n-1, [d-1, 1]})$ fills its ambient space if and only if
    \begin{itemize}
    \item[(i)] $n \in \{3, 4\}$ and $d \ge 2$, \quad or 
    \item[(ii)] $n = 5$ and $d \in\{ 2, 3, 4, 5\}$, \quad or 
    \item[(iii)] $n = 6$ and $d = 2$.  
    \end{itemize}

\item[(b)] In all other cases $\sigma_{3} (\X_{n-1, [d-1, 1]})$ is defective with dimension
\[
\dim \sigma_{3} (\X_{n-1, [d-1, 1]}) = \begin{cases} 
6 n - 16 & \text{ if } d = 2 \\[.5ex]
{\ds \binom{d+n-1}{n-1} - \binom{d+n-4}{d} + 3n -10 } & \text{ if } d \ge 3. 
\end{cases}
\]
\end{itemize}
\end{cor}

\begin{proof}

Consider first $d = 2$. Then we can apply Theorem \ref{thm:balanced}. However, it is easier to argue directly. In this case, the ideal $I$ in Lemma \ref{lem:dim for r=2} is generated by $2 \ell = 6$ linear forms. Hence $[I]_2 = [S]_2$ if and only if $n \le 6$. If $n \ge 7$, then we get 
\begin{eqnarray*}
\dim \sigma_{3} (\X_{n-1, ([1, 1]}) = \dim_{\k} [I]_2 - 1 & = & \binom{n+1}{2} - \binom{n-5}{2} - 1 \\[1ex]
& = & 6n - 16. 
\end{eqnarray*} 
Moreover, $\sigma_{2} (\X_{n-1, [1, 1]})$ is defective if it does not fill its ambient space by Theorem \ref{thm:balanced}. 

Assume now $d \ge 3$. Then Theorem \ref{thm:case d-1, 1} shows that $\sigma_{3} (\X_{n-1, [d-1, 1]})$ fills its ambient space if and only if 
\[
\frac{n}{2} \le \ell_0 \le 3.  
\]
Furthermore, $\sigma_{3} (\X_{n-1, [d-1, 1]})$ does not fill its ambient space if $n \ge 6$ by Theorem \ref{thm:dim formula}. Hence, it remains to consider the cases $n \in \{3, 4, 5\}$. 

If $n = 5$, this forces $\ell_0 = 3$, which means 
\[
d+1 \le 6 \quad \text{ and } \quad \binom{d+2}{2} > 6, 
\]
that is, $d \in \{3, 4, 5\}$. 

Since $\ell_0 \leq n-1$, we get $\ell_0 \le 3$ if $n \le 4$, and thus $\sigma_{3} (\X_{n-1, ([1, 1]})$ fills its ambient space. This shows Part~(a). Claim (b) follows by Theorem \ref{thm:case d-1, 1}. 
\end{proof}


\section{The variety of reducible forms}
     \label{sec:variety red forms} 

Every reducible form of degree $d$ in $n$ variables corresponds to a point of the variety 
\[
\X_{n-1, d} = \bigcup_{k=1}^{\lfloor \frac{d}{2} \rfloor} \X_{n-1, [d-k, k]}.
\]
  (Notice that this holds even for reducible forms with more than two factors.)  

Thus, we call $\X_{n-1, d}$ the {\em variety of reducible forms of degree $d$ in $n$ variables}. In this  section we study its secant varieties. This is based on the results on the secant varieties of the various $\X_{n-1, [d-k, k]}$, where $k$ varies between $1$ and $\lfloor \frac{d}{2} \rfloor$. 

\begin{rem}
     \label{rem:d is 2} 
The variety $\X_{n-1, d}$ is irreducible if and only if $d = 2$. In this case $\X_{n-1, 2} = \X_{n-1, [1, 1]}$ and 
\[
\dim \sigma_{\ell} (\X_{n-1, 2}) = \begin{cases} 
2 \ell (n - \ell) + \ell -1 & \text{if } 2\ell < n\\
\binom{n+1}{2} - 1 & \text{if } 2\ell \ge n. 
\end{cases}
\]  
Moreover, $\sigma_{\ell} (\X_{n-1, 2})$ is defective if and only if $2 \le \ell < \frac{n}{2}$. 
In this case, the defect is given by Theorem \ref{thm:balanced}. 
\end{rem}

We begin by determining the dimension of $\X_{n-1, d}$. The next result is an immediate consequence of 
the definition of $X_ {n - 1, d}$, of the second inequality of Corollary  \ref{cor:dim comparison}, and of formula \eqref{dim X}.

\begin{prop}
     \label{prop:dim red var} 
For all $d \ge 2$ and $n \ge 3$, 
\[
\dim \X_{n-1, d} = \dim \X_{n-1, [d-1, 1]} = \binom{d + n-2}{n-1} + n-2.      
\]
\end{prop}

Corollary \ref{cor:dim comparison} also gives  that $\X_{n-1, [d - 1, 1]}$ is the unique irreducible component of $\X_{n-1, d}$ that has the same dimension as $\X_{n-1, d}$. 

As we noted above, $\X_{n-1,d}$ is not irreducible as soon as $d > 2$.  Thus, in order to calculate the dimension of $\sigma_\ell(\X_{n-1,d})$ one must also consider the (embedded) $\ell$-joins of the irreducible components of $\X_{n-1,d}$.

Recall that if $\X_1, \cdots, \X_\ell$ are irreducible varieties in $\P^m$ (not necessarily distinct), then the {\it embedded join} of $\X_1, \cdots , \X_\ell$, denoted
$$
J(\X_1, \cdots , \X_\ell)
$$
is the Zariski closure of the union of all the linear spaces $\langle P_1, \ldots , P_\ell \rangle \subset \P^m$ where $P_i \in \X_i$.  If all the $\X_i = \X$ then this is nothing other than $\sigma_\ell(\X)$. Furthermore, for any (possibly reducible) variety $\X \subset \P^m$, the parameter count mentioned in the introduction gives 
\[
\dim \sigma_\ell(\X) \le \min \{m, \ \ell \cdot \dim \X + \ell -1\}, 
\]
and the right-hand side is called the \emph{expected dimension} of $ \sigma_\ell(\X)$.

The following Lemma shows that if $\X_1, \cdots, \X_\ell$ are any $\ell$ irreducible components of $\X_{n-1,d}$ and $2\ell \leq n$, then 
\[
\dim \sigma_\ell(\X_{n-1,[d-1,1]}) \geq \dim J(\X_1, \cdots , \X_\ell).
\]

\begin{lem}
    \label{prop:compare Hilb c.i.}
Consider integers $k_1, k_2,\ldots, k_{\ell} \in \{1,\ldots, \lfloor \frac{d}{2} \rfloor \}$, where $d \ge 2$. 
Let $I \subset S = \k[x_1,\ldots,x_n]$ be an ideal generated by $2 \ell$ general forms of degrees $k_1, k_2,\ldots, k_{\ell}, d- k_1, d -  k_2,\ldots, d- k_{\ell}$.  Let $J \subset S$ be an ideal generated by $\ell$ general linear forms 
and $\ell$ general forms of degree $d - 1$. If $2\ell \le n$, then, for all integers $j$, 
\[
\dim_{\k} [S/I]_j \ge \dim_{\k} [S/J]_j. 
\]
Moreover, if $k_i > 1$ for some $i \in \{1,\ldots,\ell\}$, then there is some $j$ such that this is a strict inequality. 
\end{lem}

\begin{proof}
Consider first the case, where $n = 2 \ell$. If $n = 2$, then 
\[
\dim_{\k} [S/J]_j = \begin{cases}
1 &  \text{if } 0 \le j \le d-2\\
0 & \text{otherwise}, 
\end{cases}
\]
and the claim follows in this case. Using Hilbert series, this observation can be expressed as 
\begin{equation}
    \label{eq:compare c.i.}
\frac{1}{(1 - t)^2} (1 - t^k) (1 - t^{d-k}) \ge \frac{1}{(1 - t)^2} (1 - t) (1 - t^{d-1}) \quad \text{whenever } 1 \le k \le \frac{d}{2}. 
\end{equation}
Let now $n \ge 4$. Then the Hilbert series of $S/I$ and $S/J$ are 
\[
\HS (S/I) =  \prod_{i = 1}^{\ell} \frac{(1 - t^{k_i}) (1 - t^{d-k_i})}{(1 - t)^2}
\]
and 
\[
\HS (S/J) = \prod_{i = 1}^{\ell} \frac{(1 - t) (1 - t^{d-1})}{(1 - t)^2}. 
\]
Thus, Inequality \eqref{eq:compare c.i.} gives 
\[
\HS (S/I) \ge \HS (S/J), 
\]
as desired. 

Finally, assume $n > 2 \ell$. Then 
\begin{align*}
\HS (S/I)  & =  \frac{1}{(1 - t)^{n - 2 \ell}} \cdot   \prod_{i = 1}^{\ell} \frac{(1 - t^{k_i}) (1 - t^{d-k_i})}{(1 - t)^2} \\
& \ge \frac{1}{(1 - t)^{n - 2 \ell}} \cdot  \prod_{i = 1}^{\ell} \frac{(1 - t) (1 - t^{d-1})}{(1 - t)^2} \\
& = \HS (S/J), 
\end{align*}
where the estimate follows from the case $n = 2 \ell$ and the fact that the coefficients in the power series expansion of $\frac{1}{(1 - t)^{n - 2 \ell}}$ are all non-negative. 
\end{proof}

We are ready for the main result of this section. 

\begin{thm}
    \label{thm:red forms, small ell}
Assume $2\ell \le n$. Then 
\[
\dim  \sigma_{\ell}  (\X_{n-1, d}) = \dim \sigma_{\ell}  (\X_{n-1, [d - 1, 1]}).  
\]    
Moreover:
\begin{itemize}

\item[(a)] The variety $\sigma_{\ell}  (\X_{n-1, d})$ fills its ambient space if and only if 
\begin{itemize}
\item[(i)] $2\ell = n$ and $d = 2$; \quad or 

\item[(ii)] $\ell = 2, \ n = 4$, and $d = 3$. 
\end{itemize}

\item[(b)]  If  $\sigma_{\ell}  (\X_{n-1, d})$ does not fill its ambient space, then its dimension is 
\[
\dim  \sigma_{\ell}  (\X_{n-1, d}) = 
\binom{d+n-1}{n-1} -   \binom{d+n- \ell - 1}{d}  + \ell (n- \ell) - 1, 
\]
and $\sigma_{\ell}  (\X_{n-1, d})$ is defective. 
\end{itemize}
\end{thm}

\begin{proof}
Let $P_1,\ldots,P_{\ell} \in \X_{n-1, d}$ be points such that each $P_i$ is a general point on some component, say $\X_{n-1, [d - k_i, k_i]}$, of $\X_{n-1, d}$.  
Then Terracini's Lemma gives (as in  Corollary \ref{dim sigma ell}) 
\[
\dim (\sigma_\ell(\X_{n-1,d} )) = \max \left \{\dim_{\k} \left [ I_{P_1} + \cdots + I_{P_\ell} \right ]_d -1 \; | \;  k_1, k_2,\ldots, k_{\ell} \in \left \{1,\ldots, \textstyle{\left \lfloor \frac{d}{2} \right  \rfloor} \right \} \right \}. 
\]
Using the notation of Lemma \ref{prop:compare Hilb c.i.}, this implies 
\[
\dim (\sigma_\ell(\X_{n-1,d})) = \dim_{\k} [J]_d - 1 = \dim (\sigma_\ell(\X_{n-1,[d-1, 1]} )), 
\]
as desired.

Part (a) is now a consequence of the second part of Theorem \ref{thm:dim formula}. We claim that part (b) follows  from Theorem \ref{thm:case d-1, 1}. Indeed, if $\ell < \ell_0$ then Theorem \ref{thm:case d-1, 1} gives the desired dimension. If $\ell_0 \leq \ell$ then Theorem \ref{thm:case d-1, 1} yields that $\sigma_\ell(\X_{n-1,[d-1,1]})$ fills its ambient space.
\end{proof} 

\begin{rem}
     \label{rem:compare secant dim r is 2}
Lemma \ref{prop:compare Hilb c.i.} also  implies that
\[
\dim \sigma_{\ell}  (\X_{n-1, [d-k, k]}) \le \dim \sigma_{\ell}  (\X_{n-1, [d-1, 1]}) \quad \text{ if } 2 \le k \le \frac{d}{2}, 
\]  
provided $\ell \le \frac{n}{2}$. We conjecture that this bound is true without the latter restriction. Notice that we have an upper bound for $\dim \sigma_{\ell}  (\X_{n-1, [d-k, k]})$ by Proposition \ref{prop:Froeberg gives main conj} and that we know   $\dim \sigma_{\ell}  (\X_{n-1, [d-1, 1]})$ by Theorem \ref{thm:case d-1, 1}. This reduces this conjecture to a comparison of two numbers. However, we have been unable to establish the needed estimate.  
\end{rem}

By Theorem \ref{thm:case d-1, 1}, we know exactly when the secant variety $\sigma_{\ell}  (\X_{n-1, [d - 1, 1]})$ fills its ambient space. This gives:

\begin{thm}
    \label{thm:secants red forms fill}  
The secant variety  $\sigma_{\ell}  (\X_{n-1, d})$ fills its ambient space if $\ell \ge \ell_0$, 
where  
\[
\ell_0 = 
\min \left \{ \ell \ge \frac{n}{2} \; | \; \ell \in \ZZ \text{ and } \binom{d - \ell + n - 1}{d} \le \ell (n- \ell) \right \}.
\]

In particular, $\sigma_{\ell}  (\X_{n-1, d})$ fills its ambient space if $\ell \ge n -1$. 
\end{thm}

\begin{proof}
Notice that 
\[
\dim  \sigma_{\ell}  (\X_{n-1, d}) \ge \dim \sigma_{\ell}  (\X_{n-1, [d - 1, 1]}). 
\]    
Thus, the claim follows from Theorem \ref{thm:case d-1, 1} if $d \ge 3$. If $d = 2$, then $\ell_0 = \lceil \frac{n}{2} \rceil$, and we conclude by Remark  \ref{rem:d is 2}. 
\end{proof}

\begin{proof}[Proof of Theorem \ref{thm:intro red forms}]
Combine Theorems \ref{thm:red forms, small ell} and \ref{thm:secants red forms fill}. 
\end{proof}

\begin{rem}
     \label{rem:conj for red forms} 
If $2\ell > n$, then Theorems  \ref{thm:red forms, small ell} and \ref{thm:secants red forms fill} do not rule out the possibility that $\sigma_\ell(\X_{n-1,d})$ fills its ambient space even when $\sigma_\ell(\X_{n-1,[d-1,1]})$ does not. However, we do not expect  this ever happening. In fact, we suspect that the following extension of Theorem \ref{thm:red forms, small ell} is true: 
\[
\dim  \sigma_{\ell}  (\X_{n-1, d}) = \dim \sigma_{\ell}  (\X_{n-1, [d - 1, 1]}), 
\]   
for all $\ell, \ n$ and $d$. 
If so, then, by Theorem \ref{thm:case d-1, 1},  the converse of Theorem \ref{thm:secants red forms fill} is true and $\sigma_{\ell}  (\X_{n-1, d})$ is defective whenever it does not fill its ambient space. 
\end{rem}


\section{Application to secant varieties of Segre varieties}\label{AOPstuff}

The paper \cite{AOP} by Abo, Ottaviani and Peterson  
classifies all Segre varieties $X$ such that the $(\ell-1)$-secant variety is defective for some $\ell < 7$
and it raises some questions as to conjecturally what happens in general. 
Our results verify certain cases of these conjectures.  

We first recall some terminology from \cite{AOP} and a related result from \cite{CGG:0}.
Assume $2\leq n_r\leq\cdots\leq n_1$ (this is consistent with our ordering convention, but it is the reverse of what \cite{AOP} does). Say that 
$(n_1,\ldots, n_r)$ is \emph{balanced} if $n_1-1\leq \Pi_{i=2}^rn_i-\sum_{i=2}^r(n_i-1)$ and 
\emph{unbalanced} otherwise. Using this terminology, 
\cite{CGG:0} proves that with respect to the Segre embedding of $X={\mathbb P}^{n_1-1}\times \cdots\times {\mathbb P}^{n_r-1}$ 
in projective space, $X$ has a defective $(\ell-1)$-secant 
variety for some $\ell$ if $(n_1,\ldots, n_r)$ is unbalanced. This result, \cite[Proposition 3.3]{CGG:0},
was paraphrased  in \cite{AOP} essentially as follows (see \cite[Lemma 4.1]{AOP}):
\vskip\baselineskip

\begin{prop} With respect to the Segre embedding of $X={\mathbb P}^{n_1-1}\times \cdots\times {\mathbb P}^{n_r-1}$ 
in projective space,  $\sigma_\ell(X)$ is defective for $\ell$ satisfying 
$$\Pi_{i=2}^rn_i-\sum_{i=2}^r(n_i-1)<\ell<\min\{\Pi_{i=2}^rn_i,n_1\}.$$
\end{prop}

By a Segre variety $X$ being \emph{unbalanced} \cite{AOP} means $X={\mathbb P}^{n_1-1}\times \cdots\times {\mathbb P}^{n_r-1}$ where
$(n_1,\ldots,n_r)$ is unbalanced. With these definitions, \cite[Question 6.6]{AOP} then asks:
\vskip\baselineskip

\begin{ques}\label{AOPquest} 
Is it true for a Segre variety $X$ such that $\sigma_\ell(X)$ is defective for some $\ell$,
that either: 

1. $X$ is unbalanced; or

2. $X={\mathbb P}^2\times{\mathbb P}^{n-1}\times{\mathbb P}^{n-1}$ with $n$ odd; or

3. $X={\mathbb P}^2\times{\mathbb P}^3\times{\mathbb P}^3$; or

4. $X={\mathbb P}^1\times{\mathbb P}^1\times{\mathbb P}^{n-1}\times{\mathbb P}^{n-1}$?
\end{ques}

The conjecture in \cite{AOP} is that the answer is yes; to rephrase:

\begin{conj}\label{AOPconj}
Let $X$ be a balanced Segre variety but not any of those listed in items 2, 3 or 4 of Question \ref{AOPquest}.
Then $\sigma_\ell(X)$ is not defective for all $\ell\geq 2$.
\end{conj}

Our results verify this in various cases, as we now explain.

It is useful to reconsider the morphism \eqref{SegreToRedFormsMorphism}.
Let $X={\mathbb P}^{n_1-1}\times \cdots\times {\mathbb P}^{n_r-1}$. The points of the Segre embedding of $X$ 
in ${\mathbb P}^{N-1}$, where $N=\Pi_i n_i$, are exactly those of the form
$[\ldots, a_{i_11}a_{i_22}\cdots a_{i_rr},\ldots]$, where $(a_{1j},\ldots,a_{n_jj})\in {\mathbb P}^{n_j-1}$.
We can regard $[\ldots, a_{i_11}a_{i_22}\cdots a_{i_rr},\ldots]$ as the multi-homogeneous polynomial
\[
\begin{split}
F(\ldots,x_{ij},\ldots)&=\sum a_{i_11}a_{i_22}\cdots a_{i_rr}x_{i_11}x_{i_22}\cdots x_{i_rr}\\
&=\Pi_{j=1}^r (a_{1j}x_{1j}+\cdots+a_{n_jj}x_{n_jj})
\in \k[\ldots,x_{ij},\ldots]
\end{split}
\]
of multi-degree $(1^r)=(1,\ldots,1)$, where the variables $x_{ij}$ are indexed by
$1\leq i\leq n_j$ with $1\leq j\leq r$ and $\k[{\mathbb P}^{n_j-1}]=\k[x_{1j},\ldots,x_{n_jj}]$. 
Now suppose that $n_j=\binom{d_j+n-1}{n-1}$ for each $j$, where $d_1\geq d_2\geq \cdots\geq d_r$ and $d=d_1+\cdots+d_r$.
Regarding the coordinate variables $x_{ij}$ of ${\mathbb P}^{n_j-1}$
as an enumeration of the monomials $M_{ij}\in \k[x_1,\ldots,x_n]$ of degree $d_j$ for each $j$, 
we get a map 
$$F\mapsto \overline{F}\in \k[x_1,\ldots,x_n],$$ 
where $\overline{F}$ is the degree $d$ singly homogeneous polynomial
$\overline{F}(\ldots,M_{ij},\ldots)$ obtained by substituting $M_{ij}$ into $x_{ij}$.
Note that the map $\overline{\vbox to.08in{\hbox to.1in{\hfil}}}$ is actually a linear homomorphism  
\[
\overline{\vbox to.08in{\hbox to.1in{\hfil}}}:(\k[\ldots,x_{ij},\ldots])_{(1,\ldots,1)}\to (\k[x_1,\ldots,x_n])_d.
\]
The projectivizations of $(\k[\ldots,x_{ij},\ldots])_{(1,\ldots,1)}$ and $(\k[x_1,\ldots,x_n])_d$
are ${\mathbb P}^{N-1}$ and ${\mathbb P}^{\binom{d+n-1}{n-1}-1}$, respectively. The restriction of $\overline{\vbox to.08in{\hbox to.1in{\hfil}}}$ 
to the affine cone corresponding to the Segre embedding of $X$ in ${\mathbb P}^{N-1}$ is just the surjective morphism induced
on affine cones by the morphism \eqref{SegreToRedFormsMorphism} of $X$ to $\X_{n-1,\lambda}$
for $\lambda=[d_1,\ldots,d_r]$.

Since $\overline{\vbox to.08in{\hbox to.1in{\hfil}}}:(\k[\ldots,x_{ij},\ldots])_{(1,\ldots,1)}\to (\k[x_1,\ldots,x_n])_d$
is linear, it maps the affine cone of the $(\ell-1)$-secant variety of $X$ to the affine cone of the $(\ell-1)$-secant variety of $\X_{n-1,\lambda}$.
If $\sigma_{\ell} (\X_{n-1,\lambda})$ is not defective,
let us say that the $(\ell-1)$-secant variety $\sigma_{\ell} (\X_{n-1,\lambda})$
of $\X_{n-1,\lambda}$ does not \emph{overly fill} its ambient space if 
$\dim \sigma_{\ell} (\X_{n-1,\lambda}) = \ell \cdot \dim \X_{n-1,\lambda} + \ell-1$. Also, given the partition $\lambda=[d_1,\ldots,d_r]$,
let $X_{n-1,\lambda}$ denote the Segre embedding of  ${\mathbb P}^{n_1-1}\times \cdots\times {\mathbb P}^{n_r-1}$, where
$n_i=\binom{d_i+n-1}{n-1}$. We thus have:

\begin{thm}\label{AOPholds}
If $\sigma_{\ell} (\X_{n-1,\lambda})$ is not defective and does not overly fill its ambient space, then
$\sigma_\ell(X_{n-1,\lambda})$ is not defective.
\end{thm}

\begin{proof}
Assume $\sigma_{\ell} (\X_{n-1,\lambda})$ does not overly fill its ambient space.
If $\sigma_\ell(X_{n-1,\lambda})$ were defective, then 
$\dim \sigma_\ell(X_{n-1,\lambda}) 
< \ell \cdot \dim X_{n-1,\lambda} + \ell-1=\ell \cdot \dim \X_{n-1,\lambda} + \ell-1$.
Since $\sigma_{\ell} (\X_{n-1,\lambda})$ does not overly fill its ambient space, we have
$\ell \cdot \dim \X_{n-1,\lambda} + \ell-1
=\dim \sigma_\ell(\X_{n-1,\lambda})$ and hence $\dim \sigma_\ell(X_{n-1,\lambda}) 
< \dim \sigma_\ell(\X_{n-1,\lambda})$.
But $\overline{\vbox to.08in{\hbox to.1in{\hfil}}}:(\k[\ldots,x_{ij},\ldots])_{(1,\ldots,1)}\to (\k[x_1,\ldots,x_n])_d$
maps the affine cone of $\sigma_\ell(X_{n-1,\lambda})$ onto the affine cone of
$\sigma_\ell(\X_{n-1,\lambda})$, so we must have
$\dim \sigma_\ell(X_{n-1,\lambda}) 
\geq\dim \sigma_\ell(\X_{n-1,\lambda})$, hence $\sigma_\ell(X_{n-1,\lambda})$ cannot be defective.
\end{proof}

As an example, we have the following corollary which verifies Conjecture \ref{AOPconj} in a range of cases. Although this particular consequence is known (see \cite[Proposition 2.3]{CGG:0}), our approach is new. 


\begin{cor}
Let $1\leq d_r\leq \cdots\leq d_1<d_2+\cdots+d_r$, $3\leq r$ and $4\leq 2\ell\leq n$ be integers with
$n_i=\binom{d_i+n-1}{n-1}$ and $a=n_1\cdots n_r$. If $X$ is the Segre embedding of 
${\mathbb P}^{n_1-1}\times\cdots\times {\mathbb P}^{n_r-1}$ in ${\mathbb P}^{a-1}$, then 
$\sigma_\ell(X)$ is not defective. Moreover, $X$ is balanced for $d_2=\cdots=d_r\gg0$.
\end{cor}

\begin{proof}
By Theorem \ref{general r geq 3}, $\sigma_\ell(\X_{n-1,\lambda})$ is not defective  for $\lambda=[d_1,\ldots,d_r]$
and does  not overly fill its ambient space, hence by Theorem \ref{AOPholds}, $\sigma_\ell(X)$ is not defective.
Now we just need to check that $X$ is balanced when $d_2=\cdots=d_r\gg0$. Since
$\binom{d_1+n-1}{n-1} < \binom{d_2(r-1)+n-1}{n-1}$, it suffices to show that
$\binom{d_2(r-1)+n-1}{n-1}\leq \binom{d_2+n-1}{n-1}^{r-1}-(r-1)\binom{d_2+n-1}{n-1}$.
But $\binom{d_2(r-1)+n-1}{n-1}$ is a polynomial in $d_2$ of degree $n-1$ while
$\binom{d_2+n-1}{n-1}^{r-1}-(r-1)\binom{d_2+n-1}{n-1}$ is a polynomial in $d_2$ of degree $(r-1)(n-1)$,
hence the inequality must hold for $d_2\gg0$.
\end{proof}


\section{Further questions and comments}

In this section we pose some natural questions arising from our work.

It is clear from the previous sections that the major algebraic question left open in the paper is the extent to which 
sums of at least two generic tangent space ideals to the varieties of reducible hypersurfaces have enough Lefschetz elements (see Conjecture \ref{conj:WLP}). One should note that less information than the full Weak Lefschetz Property is needed to establish the conjectured dimension of secant varieties to varieties of reducible forms. This allowed us to use results by Hochster and Laksov \cite{HL}, Anick \cite{anick}, and a theorem on complete intersections (see 
\cite{stanley, watanabe, RRR, sekiguchi} or \cite{HP}) in the proofs of Theorems \ref{thm:case d-1, 1} and \ref{thm:dim when froeberg}, respectively.  It would be very interesting to have new instances where (partial) Weak Lefschetz Properties are established.

As we have observed earlier in the paper, each variety of reducible hypersurfaces is a finite projection of a Segre embedding of a product of projective spaces.  Are there more geometric conclusions (than those we have found in Section \ref{AOPstuff}) that we can draw about the secant varieties of the Segre embedding from the more abundant information we have for the secant varieties of the varieties of reducible hypersurfaces? 

Another question is if any of the secant varieties of the varieties of reducible forms are arithmetically Cohen-Macaulay, apart from the trivial cases in which those secant varieties are themselves hypersurfaces in their ambient space.


Again, apart from some very small examples, we do not have equations for the varieties of reducible hypersurfaces, much less for their secant varieties (even in cases where we know the latter are hypersurfaces in their ambient spaces (see Remark \ref{hyper surf})). It would be interesting to have some intrinsic equations for these varieties or a bound on the degrees of their equations.

Mammana \cite{mammana} gives a formula for the degree of the variety of reducible plane curves but we have no generalization of that formula for varieties of reducible hypersurfaces beyond the case of plane curves.  More generally, one would like to have a formula for  the degree of the secant varieties of these varieties. This is not known even for varieties of plane curves, except in the most trivial of cases.  It would even be interesting to know the degree when the variety is a hypersurface in its ambient  space. 
 
\vskip .5cm
\thanks
\noindent{\bf Acknowledgements:}
   
The authors wish to thank Queen's University and NSERC (Canada), in the person of the second author,
for kind hospitality during the preparation of this work. 

Catalisano and Gimigliano were partially supported by GNSAGA of INDAM and by MUIR funds (Italy).  Geramita was partially supported by NSERC (CANADA) under grant No. 386080, while Harbourne was partially supported by NSA (US) under grant NO. H98230-13-1-0213.  Both Migliore and Nagel were partially supported by the Simons Foundation under grants No. 309556 (Migliore) and 317096 (Nagel).  Shin was supported by the Basic Science Research Program of the NRF (Korea) under grant No. 2013R1A1A2058240/2. 

The authors are also grateful to the referees for helpful suggestions and comments.

 
\bibliographystyle{alpha}
\bibliography{biblioSecant}

\end{document}